\documentclass{amsart}
\usepackage{verbatim,amsfonts,color,mathrsfs,stmaryrd,graphicx, mathdots}
\usepackage{xcolor}
\usepackage{tikz-cd}
\usepackage{tikz}
\usepackage{subcaption}
\usepackage{tkz-euclide}
\usetikzlibrary{patterns}   
\usetikzlibrary{positioning}
\usetikzlibrary{decorations.pathreplacing,calligraphy,decorations.markings,arrows}
\usetikzlibrary{calc}
\def\centerarc[#1](#2)(#3:#4:#5)% Syntax: [draw options] (center) (initial angle:final angle:radius)
    { \draw[#1] ($(#2)+({#5*cos(#3)},{#5*sin(#3)})$) arc (#3:#4:#5); }

\usepackage{pgfplots}

\pgfplotsset{compat=1.6}
\pgfplotsset{soldot/.style={color=blue,only marks,mark=*}} 
\pgfplotsset{holdot/.style={color=blue,fill=white,only marks,mark=*}}

\tikzset{->-/.style={decoration={
  markings,
  mark=at position #1 with {\arrow{>}}},postaction={decorate}}}
 \tikzset{-<-/.style={decoration={
  markings,
  mark=at position #1 with {\arrow{<}}},postaction={decorate}}}

\usepackage{enumitem}
\usepackage{amssymb}

\theoremstyle{plain}
\newtheorem{Thm}{Theorem}%[section]

\newtheorem{Lem}[Thm]{Lemma}
\newtheorem{Prop}[Thm]{Proposition}

\theoremstyle{definition}
\newtheorem{Def}[Thm]{Definition}
\newtheorem{Rmk}[Thm]{Remark}

\theoremstyle{remark}

\def\ZZ{{\mathbb Z}}

\def\alp{0.54}

\usepackage{hyperref}

\begin{document}

%Topmatter 

%One author
\title[A tale of two pseudo-Anosov maps]{Distinctness of two pseudo-Anosov maps}

\author{Thomas A. Schmidt}
\address{Oregon State University\\Corvallis, OR 97331}
\email{toms@math.orst.edu}

 \author{Mesa Walker }
\address{Oregon State University\\Corvallis, OR 97331}
\email{walkemes@oregonstate.edu}

\date{16 September 2024}

%End topmatter

\begin{abstract}  In 1981, Arnoux and Yoccoz gave the first examples of pseudo-Anosov maps with odd degree stretch factors.  In 1985,  D.~Fried deduced the existence of a pseudo-Anosov map in genus three with the same stretch factor as the Arnoux-Yoccoz example in that genus, and asked if these were the same.    We show that they are distinct.    We do this by, in a sense, reversing Fried's construction; we show that the mapping torus of the pseudo-Anosov map induced by the Arnoux-Yoccoz map on the surface obtained by blowing-up its two singularities  has  no cross section which is a torus with two points blown-up.   
\end{abstract}

\maketitle

\tableofcontents

\section{Introduction}  Thurston classified surface homeomorphisms into three classes, with the pseudo-Anosov homeomorphisms the main case of interest.   As we recall below, associated to a pseudo-Anosov homeomorphism is an algebraic integer, its stretch factor.   In the classic reference \cite{FLP} for Thurston's work on surface homeomorphisms, published in 1979,   the authors indicate that they had not yet seen a pseudo-Anosov map whose stretch factor is of algebraic degree greater than two.  
 In 1981, Arnoux and Yoccoz \cite{AY} gave examples of pseudo-Anosov maps with stretch factors of  degree $g$ for each $g\ge 2$.     In a recent paper, Leichti-Strenner  \cite{LiechtiStrenner} write that this family of examples is ``probably the single most widely studied" of the pseudo-Anosov maps, see \cite{ABB,  Bowman,  HLM, HooperWeiss, LPV, McPoly, Str} for some of this work.

D.~Fried, in a series of influential papers including \cite{Fr79, Fr82GeomCrossTop, Fr82FlowEqui},    used techniques of algebraic topology to study the dynamics of flows and especially of flows in 3-manifolds.    
 As an illustration of some of this,   Fried \cite{Fr85}   in 1985 deduced the existence of a pseudo-Anosov map in genus three with the same stretch factor as the Arnoux-Yoccoz example in that genus, and asked if these were the same.       We show that they are in fact distinct.

Whereas Arnoux-Yoccoz used methods  of interval exchange transformations,  Fried's approach gives the surface homeomorphism by blowing down the boundary circles of a particular cross section of the suspension flow on the mapping torus of the  map, on the square torus blown up at two points, induced by a specific toral automorphism.     If the two pseudo-Anosov maps were the same, then one could reverse Fried's procedure.  That is, the square torus with two points blown-up would appear as a cross section to the suspension flow on the mapping torus of the map induced by the Yoccoz-Arnoux  map on the blow-up of the genus three surface at the two singular points of the Yoccoz-Arnoux map.  We show that no cross section of this latter suspension flow has the Euler characteristic of a torus blow-up at two points, and thus deduce that the pseudo-Anosov maps of Arnoux-Yoccoz and of Fried cannot be the same.\\

%-----------------------
\begin{Thm}\label{t:main}   The two pseudo-Anosov maps in genus three with stretch factor equal to the largest zero of $x^3-x^2-x-1$   determined by Fried and by Arnoux-Yoccoz are distinct.   In particular,  there is no cross section to the suspension flow on the mapping torus of the blow-up of the Arnoux-Yoccox surface at its two singularities whose Euler characteristic equals to $-2$.
\end{Thm}
%-----------------------

 We follow Fried's recipe for determining rough aspects of cross sections to the suspension flow in the mapping torus of a pseudo-Anosov map,   in that we: (1)  blow-up the singularities on the Arnoux-Yoccoz surface; (2) compute Fried's multivariable Lefschetz function for the suspension flow on the mapping torus of the induced pseudo-Anosov map; (3) determine the specialization polynomials, each of which has degree giving the absolute value of the Euler characteristic of a corresponding section (and whose leading zero is the stretch factor of the pseudo-Anosov map given by the return map to the cross section).  In Lemma~\ref{l:noMinTwo} we use that information to show that there is no cross section whose  Euler characteristic equals to $-2$, and thus the theorem holds.

\subsection{Thanks}  It is a pleasure to thank Saul Schleimer and Erwan Lanneau for related conversion, as well as various participants of the S\'eminaire Rauzy  in Marseille for stimulating questions.     This paper is based upon the  Oregon State University PhD dissertation of the second-named author. 
 
 \section{Background}
 
 \subsection{Pseudo-Anosov maps as affine diffeomorphisms on translation surfaces}  The main case of the Dehn-Bers-Thurston classification of surface homeomorphisms is  that of {\em pseudo-Anosov} maps, see  \cite{Hteich2} and \cite{Lanneau} for much of the following.   Suppose that $S$ is an orientable closed real surface   of genus $g\ge 2$.    A homeomorphism  $f : S \to S$ is called {\em pseudo-Anosov}
if there exists   a pair of invariant transverse measured (singular) foliations
$(\mathcal F^u, \mu^u), (\mathcal F^s, \mu^s)$  and a real number $\lambda>1$ such that $f$ multiplies
the transverse measure $\mu^u$ (resp. $ \mu^s$) by $\lambda$ (resp. $\lambda^{-1}$). The real number $\lambda$, which Thurston \cite{ThurstonPA} showed is always an algebraic integer,  is called the {\em stretch factor} of the {\em pseudo-Anosov homeomorphism} $f$.        One can extend this definition to the case  with boundary,  see  Boyland's \cite{Boyland2} discussion of  `standard models', where now the restriction is that the Euler characteristic of $S$ be negative.

A pseudo-Anosov homeomorphism  $f$ is called {\em orientable} if  $\mathcal F^u$ or $\mathcal F^s$ is orientable  --- that is,  leaves can be consistently oriented.  (It is well known that if either of these foliations is orientable then so is the other.)   As \cite{LT} recall (see their Theorem 2.4),  a pseudo-Anosov homeomorphism $f$ is orientable if and only if its stretch factor is an eigenvalue of the standard induced action on first homology $f_*: H_1(X, \mathbb Z) \to H_1(X, \mathbb Z)$.

Hubbard-Masur \cite{HM} showed that the pair of measured foliations defines  a quadratic differential and a complex structure on $S$ so that this quadratic differential is holomorphic.   Orientability of the foliations corresponds to the quadratic differential being the square of a holomorphic 1-form (thus,  an {\em abelian differential}), say $\omega$.  Fixing base points and integrating $\omega$ along paths  defines local coordinates on $S$ (in $\mathbb C$ or $\mathbb R^2$, depending on our need),  transition functions are by translations, and the result is a {\em translation surface}, $(S, \omega)$.  (The aforementioned singularities occur at the zeros of $\omega$, and each has a cone angle which is a positive integral multiple of $2 \pi$.)  The pseudo-Anosov $f$ acts affinely with respect to the local Euclidean structure of $(f, \omega)$. It is thus an {\em affine diffeomorphism}: away from the singularities $f$ is a diffeomorphism such that  its Jacobian matrix with respect to the coordinates of the flat structure is constant.  In fact, this constant matrix has eigenvalues   $\lambda, \lambda^{-1}$.       On the other hand, any affine diffeomorphism of a translation surface whose Jacobian matrix is of this form is a pseudo-Anosov homeomorphism.

 \subsection{Blowing up points on surfaces}\label{ss:blowUps}    See \cite{BSW} for a detailed discussion of (real oriented) blow-ups of translation surfaces.  The {\em blow-up} of a translation surface $S$ at $d$ distinct points is  a surface $\hat{S}$  with $d$ boundary components,  each a circle and a   {\em collapsing map},  $\mathfrak{c}:\hat{S}\to S$ which is continuous and  is one-to-one off of these circles, while sending each of these circles to the corresponding blown-up point.   Furthermore, each  boundary component can naturally be given the measure of the cone angle at the singularity; indeed, each point on a boundary component corresponds to, say,  an outgoing direction at the corresponding singularity.   
 
 The fundamental group of a blown up surface is of course directly related to the original surface's fundamental group.
 
%-----------------------
\begin{Lem}\label{l:generateFreely}    Let $S$ be a closed orientable surface of genus $g$.  Let $P = \{p_1, \dots, p_d\}$ be a set of $d$ distinct points on $S$ and  $\mathfrak{c}:\hat{S}\to S$ be the collapsing map from the surface $\hat{S}$  obtained by blowing up the set $P$.  

Let $\{\hat{\alpha}_i\}_{1\le i\le 2g+d-1}$ be a set of closed curves on $\hat{S}$, set $\alpha_i = \mathfrak{c}(\hat{\alpha}_i)$,  and suppose  that  both
\begin{enumerate}
\item[(i)]  $\pi_1(S)$ is generated by the classes of the images 
$\{[ {\alpha}_i)]\}_{1\le i\le 2g+d-1}$, and 
\item[(ii)]  for each $1\le i \le d-1$
 there is some $r\ge 1$ and  subset 
$\{j_1, \dots, j_r\} \subset \{1, \dots, 2g + d-1\}$ such that the loop $\hat{\alpha}_{j_1}*\cdots*\hat{\alpha}_{j_r}$ is freely homotopic to the
 blown-up circle $\mathfrak{c}^{-1}(p_{i+1})$. 
 \end{enumerate}
 Then $\{[ \hat{\alpha}_1],...,  [\hat{\alpha}_{2g+d-1}]\}$ is a minimal set of free generators of $\pi_1(\hat{S})$. 
\end{Lem}
%-----------------------

\begin{proof}  The surface $\hat{S}$ deformation retracts to a homeomorphic image of $S \setminus P$.   This in turn retracts onto a wedge of $2g+d-1$ circles.  Thus,   there are  closed curves  $\beta_i, 1 \le i\le 2g + d-1$ on $S$ (all avoiding $P$) whose classes generate $\pi_1(S \setminus P)$ and whose pre-images on $\hat{S}$, say $\hat{\beta_i}, 1 \le i\le 2g + d-1$ have classes generating  $\pi_1(\hat{S})$.   Moreover, using van Kampen's theorem,  we can and do choose the $\beta_i$ such that $\pi_1(S)$ is generated by  $\{[\beta_i]\}_{1 \le i\le 2g}$ and that each $\beta_{2g+j}, 1\le j \le d-1$ is a loop about $p_{j+1}$.  In particular,   $\mathfrak{c}_*: \pi_1(\hat{S}) \to \pi_1(S)$ has kernel generated by $\{ [\beta_{2g+j}]\}_{1\le j \le d-1}$.

Certainly each $[\hat{\beta}_i], 1 \le i\le 2g$ is expressible as a word in  $\{[ \hat{\alpha}_1],...,  [\hat{\alpha}_{2g+d-1}]\}$.
Furthermore,  each of the $\hat{\beta}_{2g+k}$ is in the free homotopy class of a blown up boundary circle, and hence in light of ($ii$)  can be expressed as a word in $\{[ \hat{\alpha}_1],...,  [\hat{\alpha}_{2g+d-1}]\}$.   It follows that  this latter set also  freely generates $\pi_1(\hat{S})$.
\end{proof}

Fundamental to the use of blowing-up in the setting of translation surface in Fried's construction \cite{Fr85} is the fact that given an affine diffeomorphism on a translation surface $f: S \to S$ upon blowing up a full  periodic orbit  (or multiple such orbits),  there is a uniquely corresponding $\hat{f}: \hat{S} \to \hat{S}$.  See \cite{Boyland} for a terse overview of this for general diffeomorphisms of surfaces.   Suppose that $\mathscr O = \{p_1, \dots, p_d\}$ is a subset of $S$ which $f$  permutes and let $\mathfrak c: \hat{S}\to S$ be the result of blowing up the points in $\mathscr O$.   Since $f$ sends $S\setminus \mathscr O$ to itself, for $\hat{x}\in \hat{S}$ with $\mathfrak{c}(\hat{x}) \notin \mathscr O$,  we can and do set $\hat{f}(\hat{x}) = \mathfrak{c}^{-1}(\, f(x)\,)$.   Fix $p \in \mathscr O$ and $\hat{x} \in \mathfrak{c}^{-1}(p)$.   Then $\hat{x}$ corresponds to an outgoing direction at $p$ on $S$;   the (orientation preserving) affine diffeomorphism $f$ sends this direction to an outgoing direction at $f(p)$ and hence we let $\hat{f}(\hat{x})$ be the point on $\mathfrak{c}^{-1}(\,f(p)\,)$ that corresponds to this direction. 

Under reasonable conditions, there is an inverse operation; one can `blown down' orbits of boundary components.   Since the collapsing map is injective off of each such orbit,  in the pseudo-Anosov setting stretch factors are preserved.

\subsection{Cross sections to the flow in a mapping torus} 
The {\em mapping torus} $M_f$ of a homeomorphism $f: X \to X$ is the space $X\times \mathbb R$ modulo the relation $(x, t+1) \sim (f(x), t)$.  The map $X\times \mathbb R \to M_f$ given by $(x, t)\mapsto [x,t \pmod 1]$ is a covering space projection, whose cyclic infinite deck transformation group is generated by $\tau: (x,t) \mapsto (x,t+1)$.

 For real $t$, one has a function $\Phi_t: [(x,s)] \mapsto [(x, s+t)]$, the  {\em suspension flow} on $M_f$.   When $X$ is smooth (possibly with boundary) and $f$ a diffeomorphism,  a submanifold  $K \subset M_f$ is a {\em cross section} of the  flow if every flow line $\{(x,t)\mid t \in \mathbb R\}$ meets $K$ transversely.   The (first) {\em return map} to $K$ is $r:K\to K$ given by $r(p) = \Phi_t(p)$ with $t>0$ and minimal such that this is again a point in $K$.      In \cite{Fr79}, Fried shows that  $f$ being a pseudo-Anosov map implies that the return map $r$  to any cross section $K$ is also pseudo-Anosov.    In \cite{Fr85}, Fried states that when $f$ is orientable  and preserves the orientations of its foliations, then arguing as in \cite{Fr79} shows that $r$ is orientable.

  Given  a cross section $K$, at each point $p \in K$ we can reparametrize the flow arc from $p$ to $r(p)$ to correspond to the unit interval.  Using this, one shows that the mapping torus of the return map $r$ to $K$ is homeomorphic to $M_f$; with the slightest of abuses, we will say that they are the same.  %Fried xxxx refers to cross sections being {\em flow equivalent} xxxxx.} 

Each cross section $K$ is in particular a surface. Since $M_f$ is a three manifold, there is a corresponding integral first cohomology class $u_K$. In particular, $u_K$ assigns a positive integer to the homology classes of each of the closed flow orbits of $M_f$.   Indeed, Fried \cite{Fr82GeomCrossTop}   shows that positivity of an integral $u$ on all such classes implies that $u=u_K$ for some section $K$.

\subsection{Mapping torus, Wang sequence, and Fried's Theorem}\label{ss:toFriedsTheorem}   The mapping torus of a homeomorphism $f: S \to S$ gives a fibration over $S^1$, giving the exact  `Wang sequence' 
\[ 0 \to H_1(S; \mathbb Z) \xrightarrow{f_*-\text{Id}} H_1(S; \mathbb Z)  \to H_1(M_f; \mathbb Z) \to \mathbb Z \to 0, \]
from which $H_1(M_f; \mathbb Z) \cong \text{Coker}(f_*-\text{Id})\oplus \mathbb Z$ and then certainly
\begin{equation}\label{e:freeSplit} H_1(M_f; \mathbb Z)/\text{Torsion} \cong \text{Coker}(f_*-\text{Id})/\text{Torsion}\oplus \mathbb Z.
\end{equation} 
 
We remark that we follow the terminology of \cite{Fr85}; Serre proved and named, in his dissertation \cite{S},  the Wang long exact sequence  only for fibrations over $S^n, n\ge 2$.  That said, the sequence in the case of $S^1$ has been long used, and can be verified say by specializing Hatcher's [\cite{Hatcher}, Example 2.48] application of generalized Mayer-Vietoris sequences.

Let $Q = Q_f =  \text{Coker}(f_*-\text{Id})/\text{Torsion}$, and let $S_Q$ be the covering space of $S$ corresponding to  the kernel of the composite group homomorphism $\pi_1(S)\to H_1(S;\mathbb Z)\to \text{Coker}(f_*-\text{Id}) \to \text{Coker}(f_*-\text{Id})/\text{Torsion}$.   Thus, the deck transformation group of $S_Q$ over $S$ is (isomorphic to) $Q$.    Any $f':S'\to S'$ homotopic to $f$ is such that $S'$ has a corresponding cover, also with (an isomorphic copy of) $Q$ as deck transformation group; denote this cover as $S'_Q$.

%----------------------------------------------
\begin{figure}%\scalebox{.8}{
\begin{tikzcd}[column sep=2pc,row sep=2pc]
S_{\mathcal A}\arrow{dd}[swap]{H_1(S; \mathbb Z)}\arrow{dr}{\ker \psi_Q}&
&&&&M_{f}^{^{\text{fab}}}\cong S_Q \times \mathbb R\arrow{dd}[swap]{H_1(M_f; \mathbb Z)/\text{Tor} =: G}\arrow{dr}{Q}&\\
&S_Q\arrow{dl}{Q}\arrow[hook]{rrrru}&&&&&M'_f\cong S \times \mathbb R\arrow{dl}{\langle \tau \rangle \cong \mathbb Z}\\
S\arrow[hook]{rrrrr}&&&&&M_f&
\end{tikzcd}
%}
\caption{ Here $S_{\mathcal A}$ denotes the   universal abelian covering space of $S$ and $M_{f}^{^{\text{fab}}}$  the universal free abelian covering space of $M_f$.  (The latter notation is borrowed from \cite{Parlak}.)   The group $Q =   \text{Coker}(f_*-\text{Id})/\text{Torsion}$ is the image of a group homomorphism $\psi_Q:H_1(S; \mathbb Z) \to Q$. 
The Wang sequence \eqref{e:freeSplit} shows that $Q$ is also the deck transformation group for some  intermediate covering space of $M_f$, here $M_{f}^{^{\text{fab}}}/M'_{f}$. }
\end{figure}
%----------------------------------------------

Let $G = G_f = H_1(M_f; \mathbb Z)/\text{Torsion}$.  Any cohomology class $u \in H^1(S; \mathbb Z)$ gives a $\mathbb Z$-valued homomorphism  on $G$.  Given $\zeta$ any quotient of elements of the group ring $\mathbb Z[G]$, the class $u$ defines  the rational function $p_u(t)$ by replacing each term $a g$ where $a\in \mathbb Z, g \in G$ appearing in $\zeta$ by $a t^{u(g)}$.  We call $p_u(t)$ the {\em $u$-specialization} of $\zeta$.\\

     Following \cite{Liu2}, given a covering space  $\pi: Y \to X$ and a map $f: X \to X$, we call any $F:Y \to Y$ such that $\pi\circ f = F\circ \pi$ 
an {\em elevation} of $f$ to $Y$. (This disambiguates the common usage of ``lift" for such an $F$.)  A version of the following is found by combining various statements in \cite{Fr85}.

%-----------------------
\begin{Thm}\label{t:frieds}[Fried's Theorem]   Suppose that $f:S \to S$ is an orientable pseudo-Anosov homeomorphism  and that   $f':S'\to S'$ is a cellular map homotopic to $f$.  Let $Q = \text{Coker}(f_*-\text{Id})/\text{Torsion}$, with $S_Q$ the cover of $S$ as above, and $S'_Q$ the corresponding cover of $S'$.    Also let $\mathcal F_i$ be the action of  $f'_Q$ on the $i$-cells of $S'_Q$ as $\mathbb Z[Q]$-modules, $i\ge 0$.  Set 
\[ \zeta_{Q,f}(\tau) =  \prod_{i\ge 0} \, \det( \emph{Id}- \tau \mathcal F_i)^{i+1},\]  
choose an isomorphism of the group $Q \oplus \langle \tau \rangle$ with $G$ compatible with \eqref{e:freeSplit} and apply this so as to view $\zeta_{Q,f}(\tau) = \dfrac{1+p}{1+q}$ as a quotient of elements of the group ring $\mathbb Z[G]$.   

Then  an integral cohomology class $u\in H^1(M_f; \mathbb Z)$ corresponds to a section $K= K_u$ of the suspension flow on $M_f$ if and only if for every  $g \in \mathbb Z[G]$ appearing in either $p$ or $q$ with nonzero integral coefficient,  the cohomology class $u$ assigns a positive value to  $g$.     

Furthermore, when this holds,  the rational function in the single variable $t$ given by the $u$-specialization of $\zeta_{Q,f}(\tau)$ is of degree equal to $-\chi(K)$ and has its leading zero equal to the stretch factor of the pseudo-Anosov return map to $K$.  
\end{Thm}
%-----------------------

The proof of the theorem given in \cite{Fr85} is based upon several papers of Fried.  For example, the proof of the positivity condition is given in \cite{Fr82FlowEqui}; it relies on an approach of Bowen. For an exposition of that approach, see  Shub's chapter on Markov partitions in \cite{Shub}.  In Subsection~\ref{s:faceThePositive}, we discuss aspects of the theorem in light of much more recent work of McMullen and of Liu. 

\subsection{Fibered faces}\label{ss:fibFace}  We first briefly recall some results of Thurston \cite{ThurstonNorm}.  Let $M$ be an orientable, connected 3-manifold with either empty or toroidal boundary.    For a surface $S \subset M$ with connected components $S_1, S_2, . . . , S_k$, let $\chi_{-}(S) = \sum_{i=1}^{k}\,\max( -\chi(S_i), 0)$, where $\chi$ denotes Euler characteristic.      For any $u\in H^1(M;\mathbb Z)$, there is a dual properly embedded surface $S$.  (Recall that `properly embedded' means that the inclusion map  is a diffeomorphism and the image meets the  boundary  --- if any --- nicely.) 
The {\em Thurston  norm} of $u$ is   
\[||u||_{\text{Th}}  = \min\{ \chi_{-}(S) \mid S\; \text{is properly embedded and dual to } u\}.\] 
This is then extended linearly to be defined on all of $H^1(M;\mathbb R)$.  In general, this is only a semi-norm. It is, however, a norm when $M$ is hyperbolic.  Furthermore Thurston \cite{ThurstonPA}  showed that $M$ is hyperbolic if and only if it is the mapping torus of a pseudo-Anosov homeomorphism.    The unit ball $\mathscr B_{Th}$ of the Thurston norm is a polytope, symmetric about the origin.  A top dimensional face $F$ of this ball is called {\em fibered} if  there is a $u \in  H^1(M;\mathbb Z) \cap \,\mathbb R^+\cdot F$ such that $u$ is dual to the fiber $S$ of a fibration $p:M \to S^1$.  One then also says that $u$ is fibered.  When $F$ is fibered,  every integral $u$ in $\mathbb R^+\cdot F$ is fibered.   The dual unit ball in $H_1(M_f; \mathbb Z)$ is then also a polytope, with $F$ dual to a distinguished vertex.  Thus,  when $M= M_f$ is a mapping torus of a pseudo-Anosov homeomorphism,  there is a distinguished fibered face $F_f$ of its Thurston ball, and a dual vertex $v_f$   such for any $u$ as above, $||u||_{\text{Th}}  = u(v_f)$. 

Throughout various works, Fried's perspective is that the   integral $u$   in $\mathbb R^+\cdot F_f$ correspond exactly to the sections of the suspension flow on $M_f$.  These cohomology classes are in fact then characterized by being positive on the homology classes of the closed orbits of the flow;   see \cite{Landry} for details in the case where $M_f$ has boundary.   Fried's Theorem, Theorem~\ref {t:frieds}, then characterizes these  $u$ in terms of positivity on a finite set.
 
\subsection{Fried's Theorem in light of more recent works}\label{s:faceThePositive}    Suppose that $f:S\to S$ is an orientable pseudo-Anosov homeomorphism, and that $Q = Q_f =  \text{Coker}(f_*-\text{Id})/\text{Torsion}$ has rank at least one.  In the case of $S$ being a closed surface, Liu \cite{Liu}, as he states, mainly following Fried's works,  uses \cite{T} to show that Fried's $\zeta_{Q,f}(\tau)$ equals $\Delta^{\sharp}_{M_f}$, the multivariable Alexander polynomial of $M_f$ with respect to its suspension flow.     This latter is defined to be a generator of a certain principal ideal in $\mathbb Z[G]$ --- one of the key reasons for using $G$ instead of $H_1(M_f; \mathbb Z)$ is that  the group ring $\mathbb Z[G]$ is a unique factorization domain --- and hence is only well-defined up to multiplication by units, thus by elements of the form $\pm h$ with $h \in G$.    

  Note that Liu's restriction to closed $S$ implies that also  $M_f$ is without boundary;  this allows him to invoke [\cite{T}, Theorem~14.12]   to identify $\Delta^{\sharp}_{M_f}$.  When passing from $f:S\to S$ to 
  $\hat{f}:\hat{S} \to \hat{S}$ by blowing up a finite number of $f$-orbits, the resulting $M_{\hat{f}}$  has a boundary formed by a union of tori.  One can again argue as does Liu but then rather invoke 
 [\cite{T}, Corollary~11.9] to find that  $\zeta_{Q,\hat{f}}(\tau) = \Delta^{\sharp}_{M_{\hat{f}}}$ when $Q = Q_{\hat{f}}$ has rank at least one.

Fried argues that his zeta function is a homotopy invariant by reliance on notions of  generalized Lefschetz numbers of Nielsen fixed point theory,  see \cite{Jiang}.  The homotopy invariance of Alexander polynomials is usually shown by a reliance on Reidemeister torsion, see \cite{T}.   While  Fried's Theorem, Theorem~\ref {t:frieds},  gives $\zeta_{Q,f}(\tau)$ as a  quotient of elements of $\mathbb Z[G]$, by 
Liu's result,  this quantity is always expressible as a single element of $\mathbb Z[G]$.   

McMullen's  \cite{McPoly} Alexander norm is defined by first expressing 
 $\Delta^{\sharp}_{M_f} = \sum_{\alpha}\, a_{\alpha} g_{\alpha}$ with integral  $a_{\alpha}$ and each $g_{\alpha} \in G$, and then letting $||u||_A = \max u(g_{\alpha} -  g_{\alpha'})$,   with the maximum taken over all $\alpha, \alpha'$ such that $a_{\alpha}a_{\alpha'}\neq 0$.  McMullen  showed  that the unit ball here also is a polytope and, when  $f$ is orientable preserving the orientations of its invariant foliations and $M_f$ has Betti number $b\ge 2$, that the fibered face $F_f$ of the Thurston ball is also a face of the Alexander norm ball.   The Alexander polynomial is symmetric:  There exists some $h \in G$ such that  $\Delta^{\sharp}_{M_f} = \pm h \, \sum_{\alpha}\, a_{\alpha} g^{-1}_{\alpha}$, see McMullen's \cite{McPoly} arguments of symmetry in the setting of his Teichm\"uller polynomial.  McMullen [\cite{McPoly}, proof of Theorem~7.1]  also shows that,  for $f$ as above, $u$-specializations of $\Delta^{\sharp}_{M_f}$  have the stretch factor of the return map to $K_u$ as a zero.  (Parlak \cite{Parlak} points out that this is related to an assertion of  Milnor of the 1960s.) 
 
 A second advantage of using the quotient $G$ instead of $H_1(M_f; \mathbb Z)$ itself  is that upon choosing $b$ generators for $G$
 there is an explicit ring isomorphism  $\phi: \mathbb Z[G] \to \mathbb Z[t_{1}^{\pm 1}, \dots, t_{b}^{\pm 1}]$,  to the ring of Laurent polynomials in $t_1, \dots, t_b$.   For this, one defines $\phi$ by extending in the natural manner the map sending the $j^{\text{th}}$ generator to $t_j$. Let $\xi(g) = (e_1, \cdots, e_b)$ be the exponents  of $\phi(g) = \sum\, t_{i}^{e_i}$ for any $g \in G$.    For any element $\Delta \in \mathbb Z[G]$,  its {\em Newton polytope}   $\mathcal N(\Delta)$ is the convex hull in $H_1(M; \mathbb R) \cong \mathbb R^b$ of the    $\xi(g)$  taken over those $g \in G$ appearing in $\Delta$ with nonzero coefficient.   
 
 Liu observes that  Thurston's dual vertex $v_f$,   mentioned in Subsection~\ref{ss:fibFace}, is a vertex of   $\mathcal N(\Delta^{\sharp}_{M_f})$ and that every integral  $u \in \mathbb R^+\cdot F_f$ is positive on any homology class contained in the translate to the origin of the cone of rays emanating from $v_f$ that `point into'  $\mathcal N(\Delta^{\sharp}_{M_f})$  --- see [\cite{Liu}, proof of Prop.~3.1].     Due to the convexity of $\mathcal N(\Delta^{\sharp}_{M_f})$,    Liu's translated cone is the translate of the cone whose extremal rays pass through the neighboring vertices of $v_f$.   Thus, just as in  Fried's Theorem,Theorem \ref{t:frieds}, one has that $u$ corresponds to a section if and only if it is positive on a specific finite set.

\subsection{Terse review of Fried's \cite{Fr85} example}      Fried's \cite{Fr85}  deduction of a pseudo-Anosov map sharing properties with the Arnoux-Yoccoz example in genus three was by way of showing the existence of a pseudo-Anosov return map on a genus three cross section to the flow in the mapping torus of the blow-up in two fixed points for the toral automorphism of the square torus given by his matrix $A^2 = \begin{pmatrix}5&2\\2&1\end{pmatrix}$.    In the paper,  Fried sketched the calculation of the multivariable Lefschetz zeta function $\zeta_{Q, f}$ in that setting, and determined the cohomology class $u$ such that the $u$-specialization has degree six and leading eigenvalue equal to the stretch factor of the earlier Arnoux-Yoccox example.  He then showed that the corresponding cross section  has two boundary components,  blowing these down results in the pseudo-Anosov map in question.

\section{The blown-up Arnoux-Yoccoz surface}

 \subsection{The Arnoux-Yoccoz  translation surface}

%-------------------------------------------------------------------------------------
\begin{figure}[h]
\scalebox{1.1}{
\begin{tikzpicture}[x=3.3cm,y=3.3cm] 
%make outline
\draw   (0,0)--(0,1)--(\alp, 1)--(\alp, 0.839) --(0.839, 0.839)--(0.839, \alp)--(1, \alp) -- (1,0) --cycle; 
%mark blue singularity
 \foreach \x/\y in {0.191/0.352, 0.420/0.648, 0.771/0.191%
} { \node at (\x,\y) [blue]{$\bullet$}; } 
%mark red singularity
 \foreach \x/\y in {0.044/0, 0.339/0, 0.5/0, 0.272/1, 0.691/0.839, 0.920/0.54%
} { \node at (\x,\y) [red]{$\bullet$}; } 
% slits to blue singularities as thin lines (may want to use colors to help show identifications
\draw[thin] (0.191,0)--(0.191, 0.352);        
\draw[thin] (0.420, 0)--(0.420, 0.648); 
\draw[thin] (0.771, 0)--(0.771, 0.191); 
%           identifications  horizontal sides 
%  1
\draw[thin] [green](0.228,1)--(0.228,1.04); 
\node at (0.11, 1.04){\tiny{$1$}};
%\draw[thin] [green](0.772,0)--(0.772,-0.04); 
\node at (0.88, -0.04) {\tiny{$1$}};
%  2
\node at (0.25,1)[pin={[pin edge=<-, pin distance=12pt]90:{\tiny{$2$}}}] {};
\node at (0.02,0)[pin={[pin edge=<-, pin distance=12pt]270:{\tiny{$2$}}}] {};
%  3
\node at (0.4, 1.04){\tiny{$3$}};
\node at (0.63, -0.04) {\tiny{$3$}};
%   4  
\node at (0.61, 0.879) {\tiny{$4$}};
\node at (0.12, -0.04) {\tiny{$5$}};
%   5 
\node at (0.76, 0.879) {\tiny{$5$}};
\node at (0.27, -0.04) {\tiny{$4$}};
%   6 
\node at (0.88, 0.58) {\tiny{$6$}};
\node at (0.46, -0.04) {\tiny{$6$}};
%   7
\node at (0.96, 0.58) {\tiny{$7$}};
\node at (0.39, -0.04) {\tiny{$7$}};
%  identifications  for vertical sides 
%      A
\node at (1.04, 0.27) {\tiny{$A$}};
\node at (-0.04, 0.27) {\tiny{$A$}};
\draw[thin] [green](-0.04,\alp)--(0,\alp); 
%      B
\node at (0.879, 0.7) {\tiny{$B$}};
\draw[thin] [green](0.420,0.296)--(0.460,0.296); 
\node at (0.46, 0.15) {\tiny{$B$}};
%      C
\node at (0.54,0.95)[pin={[pin edge=<-, pin distance=12pt]30:{\tiny{$C$}}}] {};
%\node at (0.61, 0.879) {\tiny{$C$}};
\draw[thin] [green](0.191,0.161)--(0.231,0.161); 
\node at (0.231, 0.08) {\tiny{$C$}};
%      D
\node at (-0.04, 0.7) {\tiny{$D$}};
\draw[thin] [green](0.420,0.456)--(0.38,0.456); 
\node at (0.37, 0.28) {\tiny{$D$}};
%      E
\node at (0.14, 0.175) {\tiny{$E$}};
%\draw[thin] [green](0.420,0.457)--(0.460,0.457); 
\node at (0.465, 0.46) {\tiny{$E$}};
%      F
\node at (0.37, 0.55) {\tiny{$F$}};
\node at (0.81, 0.1) {\tiny{$F$}};
%      G
\node at (0.231, 0.25) {\tiny{$G$}};
\node at (0.73, 0.1) {\tiny{$G$}};
\end{tikzpicture}
}
\caption{The slitted polygon $\mathcal P$. Segments $A, B, C$ of lengths $\alpha, \alpha^2, \alpha^3$ where  $\alpha$ be the positive real root of $p(x) = x^3+x^2+x-1$.    See (\cite{Hteich2}, Fig.~8.3.12) for other segment lengths; see also (\cite{HLM}, Fig.~2).   The edge identifications are by translation, and result in the Arnoux-Yoccoz \cite{AY} $g=3$ translation surface, $X$.  Note that each of the set of the red points and of the blue points is identified to a single point of  cone angle $6\pi$. }
\label{f:hubbRepn}
\end{figure}
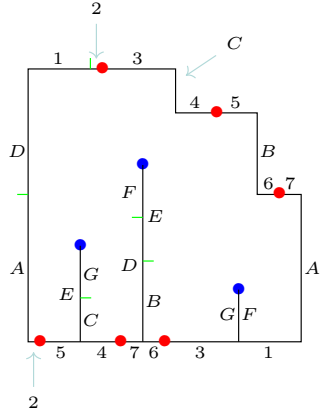
%------------------------------------------------------------------------------------- 

%-----------------------
\begin{Def}  Fix $\alpha$ as the positive real root of $p(x) = x^3+x^2+x-1$. Let $\mathcal P$ be the slitted planar polygon of Figure~\ref{f:hubbRepn}.  The Arnoux-Yoccoz surface, $X$, is the surface given by making the identifications indicated on that figure.   Define the following subregions of $\mathcal P$:

\begin{align*}
    R_1&= \{(x,y)\in \mathcal P \,\vert\,  \frac{\alpha+\alpha^2}{2}<x\leq 1,~ 0\leq y\leq \alpha\},\\ 
    R_2&=\{(x,y)\in \mathcal P \,\vert\,   \alpha<y\leq 1\},\\
   % R_3&=\{(x,y)\in \mathcal P \,\vert\,   0\leq x\leq \frac{\alpha+\alpha^2}{2},~\alpha< y\leq 1\},\\
    R_3&=\{(x,y)\in \mathcal P \,\vert\, 0\leq x\leq \frac{\alpha+\alpha^2}{2},~0\leq y\leq \alpha \}.
\end{align*}

Let $h: \mathcal P \to \mathcal P$ be given piecewise by 
\[
h(x,y)=\begin{cases}
        (\alpha x, \alpha^{-1} y)- (\frac{\alpha-\alpha^4}{2},0) &~\text{if}~ (x,y)\in R_1,\\
        (\alpha x, \alpha^{-1} y)+ (\alpha,-1) &~\text{if}~ (x,y)\in R_2,\\
        %(\alpha x, \alpha^{-1} y)+ (\alpha,-1)&~\text{if}~ (x,y)\in R_3,\\
       (\alpha x, \alpha^{-1} y)+ (\frac{\alpha-\alpha^4}{2},0)&~\text{if}~ (x,y)\in R_3\,.
   \end{cases}
\]
\end{Def}
%-----------------------

It is easily checked that the above expresses the Arnoux-Yoccoz map in explicit coordinates as an affine diffeomorphism.
%-----------------------
\begin{Lem}   The map $h: \mathcal P \to \mathcal P$ respects the identifications and hence induces a homeomorphism, also denoted by $h$, on $X$.  This map $h$ agrees with the $g=3$ pseudo-Anosov of stretch factor $\lambda = \alpha^{-1}$ of Arnoux-Yoccoz \cite{AY}. 
\end{Lem}
%-----------------------

In what follows, we often informally identify $\mathcal P$ with $X$.  
%--------------------------------- Figure Hubbard Fig 8.3.12 partitioned with `cylinders' of map;  and its image  ----------------------------------------------
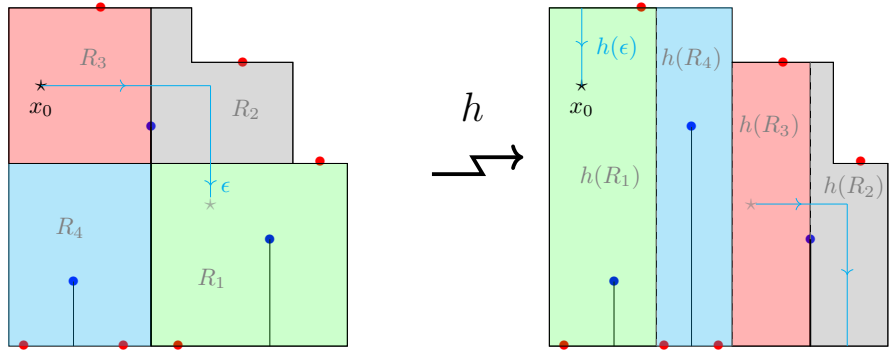
\begin{figure}[h]
\scalebox{0.9}{
\noindent
\begin{tabular}{lcr}
%--------------------------------------------- Initial `surface'
\begin{tikzpicture}[x=5cm,y=5cm] 
%make outline
\draw   (0,0)--(0,1)--(\alp, 1)--(\alp, 0.839) --(0.839, 0.839)--(0.839, \alp)--(1, \alp) -- (1,0) --cycle; 
%mark blue singularity
% \foreach \x/\y in {0.191/0.191, 0.420/0.648, 0.771/0.316%} { \node at (\x,\y) [blue]{$\bullet$}; } 
%fixed?
 \foreach \x/\y in {0.191/0.352, 0.420/0.648, 0.771/0.191%
} { \node at (\x,\y) [blue]{$\bullet$}; } 
%mark red singularity
 \foreach \x/\y in {0.339/0, 0.5/0, 0.272/1, 0.691/0.839, 0.920/0.54, 0.044/0%
} { \node at (\x,\y) [red]{$\bullet$}; } 
% slits to blue singularities as thin lines (may want to use colors to help show identifications
\draw[thin] (0.191,0)--(0.191, 0.352);        
\draw[thick] (0.420, 0)--(0.420, 0.648); 
\draw[thin] (0.771, 0)--(0.771, 0.191); 
%marking the regions R_i 
%                two of them
\draw[opacity = 0.1, fill=cyan, fill opacity = 0.3]   (0,0)--(0, \alp)--(0.420, \alp)--(0.420, 0)--cycle;
%\draw[fill=red, fill opacity = 0.3]   (0, \alp)--(0.420, \alp)--(0.420, 1)-- (0,1)--cycle;
\draw[fill=green, fill opacity = 0.2]   (0.420, 0)--(0.420, \alp)--(1, \alp)--(1, 0) -- cycle;
\draw[fill=red, fill opacity = 0.3]   (0, \alp)--(0.839, \alp) -- (0.839, 0.839) --(\alp, 0.839)--(\alp, 1)--(0,1)--cycle;%(0.420, 1)-- (\alp, 1)--(0,1) -- (0, \alp)--cycle;%(\alp, 0.839) -- (0.839, 0.839) -- (0.839, \alp)--cycle;
%label three! regions 
\node at (0.2, 0.45) [gray] {$R_3$};
\node at (0.35, 0.87) [gray] {$R_2$};
%\node at (0.7, 0.7) [gray] {$R_2$};
\node at (0.6, 0.2) [gray] {$R_1$};
%Basepoint, in order 2 orbit 
\node at (0.0957439,0.771845) [black]{$\star$};
\node at (0.0957439,0.7) [black]{$x_0$};
%\epsilon
\draw[thin,->-=0.5, cyan] (0.095,0.77)--(0.595744, 0.77);
\draw[thin,->-=0.9, cyan] (0.595744, 0.77)--(0.595744, 0.44);
\node at (0.64,0.47) [cyan]{$\epsilon$};
\node at (0.595744,0.419643) [gray, fill opacity = 0.6]{$\star$};
% \foreach \x/\y in {0.0957439/0.771845, 0.595744/0.419643%
%} { \node at (\x,\y) [black]{$\star$}; } %(\frac{\alpha-\alpha^2}{2(1+\alpha^2)},\frac{\alpha^{-1}}{\alpha^{-2}-1})
\end{tikzpicture}
&%----------------------------------------- just a labelled arrow
\begin{tikzpicture}[x=1.5cm,y=5cm] 
\node at (0, 0) {\phantom{here}};
\draw[->, ultra thick] (0.3, 0.5)--(0.8, 0.5) -- (.7, 0.55) -- (1.2,0.55);
\node at (.7, .7) {\huge{$h$}};
\end{tikzpicture}
&
%--------------------------------------- Image version 
\begin{tikzpicture}[x=5cm,y=5cm] 
%make outline
\draw   (0,0)--(0,1)--(\alp, 1)--(\alp, 0.839) --(0.839, 0.839)--(0.839, \alp)--(1, \alp) -- (1,0) --cycle; 
%mark blue singularity
% \foreach \x/\y in {0.191/0.191, 0.420/0.648, 0.771/0.316%
%} { \node at (\x,\y) [blue]{$\bullet$}; } 
%Fixed?
 \foreach \x/\y in {0.191/0.352, 0.420/0.648, 0.771/0.191%
} { \node at (\x,\y) [blue]{$\bullet$}; } 
%mark red singularity
 \foreach \x/\y in {0.339/0, 0.5/0, 0.272/1, 0.691/0.839, 0.920/0.54, 0.044/0%
} { \node at (\x,\y) [red]{$\bullet$}; } 
% slits to blue singularities as thin lines (may want to use colors to help show identifications
\draw[thin] (0.191,0)--(0.191, 0.352);        
\draw[thick] (0.420, 0)--(0.420, 0.648); 
\draw[thin] (0.771, 0)--(0.771, 0.191);  
%marking the images of the regions R_i 
\draw[opacity = 0.1, fill=cyan, fill opacity = 0.3]    (0.316,0)--(0.316, 1)--(\alp, 1)--(\alp, 0)--cycle;
%\draw[fill=red, fill opacity = 0.3]   (\alp, 0)--(\alp, 0.839) --(0.771, 0.839)-- (0.771, 0)--cycle;
\draw[fill=green, fill opacity = 0.2]  (0, 0)--(0, 1)--(0.316, 1)--(0.316,0) -- cycle;
\draw[fill=red, fill opacity = 0.3]    (\alp, 0)--(1,0)--(1, \alp)--(0.839, \alp)-- (0.839, 0.839)--(\alp, 0.839)--cycle;%(0.771, 0)--(0.771, 0.839)-- (0.839, 0.839) -- 
%\draw[thin, dashed] [gray] (0.316,0)--(0.316, 1);   
%\draw[thin, dashed] [gray] (\alp, 0)--(\alp, 0.839); 
%\draw[thin, dashed] [gray] (0.771, 0.316)--(0.771,  0.839); 
%label four regions 
\node at (0.18, 0.5) [gray] {$h(R_1)$};
\node at (0.42, 0.85) [gray] {$h(R_3)$};
\node at (0.7, 0.65) [gray] {$h(R_2)$};
%\node at (0.9, 0.47) [gray] {$h(R_2)$};
%Basepoint, in order 2 orbit 
\node at (0.0957439,0.771845) [black]{$\star$};
\node at (0.0957439,0.7) [black]{$x_0$};
%image of Basepoint, in order 2 orbit 
\node at (0.595744,0.419643) [gray, fill opacity = 0.6]{$\star$};
%image of path connecting them
\draw[thin,->-=0.5, cyan] (0.61,0.42)--(0.88,0.42);
\draw[thin,->-=0.5, cyan] (0.88,0.42)--(0.88,0);
\draw[thin,->-=0.5, cyan] (0.0957439,1)--(0.0957439,0.77);
\node at (0.2,0.88) [cyan]{$h(\epsilon)$};
\end{tikzpicture}
% $$
%
\end{tabular}
}
\caption{The Arnoux-Yoccoz map $h: X \to X$ as an affine diffeomorphism.  The map is given by  stretching vertically by a factor of $1/\alpha$ and contracting horizontally by a factor of $\alpha$; thereafter the various image pieces are translated into position.   Basepoint $x_0$ is marked with a dark star, it is an element of the unique orbit of length two of the map; its image (and preimage) is marked with a light star, see Lemma~\ref{l:orbit2}.  Paths joining these two points,  $\epsilon$ and its image under $h$, both here in cyan,  are used in the proof of Proposition~\ref{p:hatImages}.    %{\color{red} NB This is not the usual orientation for pA:  should stretch horizontally and contract vertically.   However, this is just the inverse, so no problem.  This agrees with Arnoux's presentation of the map, see [\cite{A}, Figure~6.]}
}%
\label{f:ayAsAffDiffeo}%
\end{figure}
%-------------------------------------------------------------------------------------------
 
 Elementary calculations verify the following, see Figure~\ref{f:ayAsAffDiffeo}.
%-----------------------
\begin{Lem}\label{l:orbit2}   The map $h$ fixes each of the singularities and has no other fixed point.  Furthermore,   $h$ has a unique orbit of length two, to which $x_0 =   (\frac{\alpha-\alpha^2}{2(1+\alpha^2)},\frac{1}{\alpha^{-1}-\alpha})$ belongs.
\end{Lem}
%----------------------- 

%-------------------------------Homology generators on zippered rectangle ---------------------------------------------------
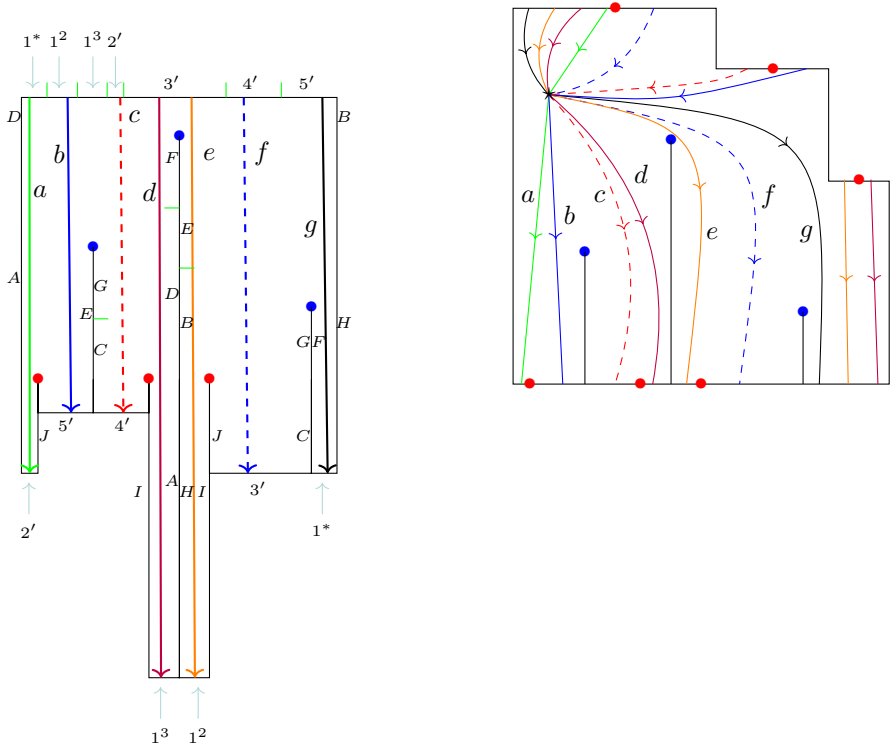
\begin{figure}[h]
\begin{tabular}{lcr}
\scalebox{1}{
%--------------------------------------------- Initial `surface'
\begin{tikzpicture}[x=5cm,y=5cm] 
%make outline
\draw   (0,0)--(0,0.75)--(0.839,0.75)--(0.839,-0.25) -- (0.771,-0.25)--(0.771,0)--(0.771,-0.25)--(0.5,-0.25)--(0.5,0)--(0.5,-0.794)--(0.42,-0.794)--(0.42,0)--(0.42,-0.794)--(0.339,-0.794)--(0.339,0)--(0.339,-0.089)--(0.191,-0.089)--(0.191,0)--(0.191,-0.089)--(0.044,-0.089)--(0.044,0)--(0.044,-0.25)--(0,-0.25)--cycle; 
%mark blue singularity
 \foreach \x/\y in {0.191/0.352, 0.420/0.648, 0.771/0.191%
} { \node at (\x,\y) [blue]{$\bullet$}; } 
%mark red singularity
 \foreach \x/\y in {0.044/0, 0.339/0, 0.5/0
} { \node at (\x,\y) [red]{$\bullet$}; } 
% slits to blue singularities as thin lines 
\draw[thin] (0.191,0)--(0.191, 0.352);        
\draw[thin] (0.420, 0)--(0.420, 0.648); 
\draw[thin] (0.771, 0)--(0.771, 0.191); 
\draw[thin] [green](0.068,0.75)--(0.068,0.79); 
\draw[thin] [green](0.149,0.75)--(0.149,0.79); 
\draw[thin] [green](0.228,0.75)--(0.228,0.79); 
\draw[thin] [green](0.272,0.75)--(0.272,0.79); 
\draw[thin] [green](0.544,0.75)--(0.544,0.79); 
\draw[thin] [green](0.691,0.75)--(0.691,0.79); 
%new nodes
%1s
\node at (0.03,0.75)[pin={[pin edge=<-, pin distance=12pt]90:{\tiny{$1^*$}}}] {};
\node at (0.8,-0.25)[pin={[pin edge=<-, pin distance=12pt]270:{\tiny{$1^*$}}}] {};
\node at (0.1,0.75)[pin={[pin edge=<-, pin distance=12pt]90:{\tiny{$1^2$}}}] {};
\node at (0.47,-0.794)[pin={[pin edge=<-, pin distance=12pt]270:{\tiny{$1^2$}}}] {};
\node at (0.19,0.75)[pin={[pin edge=<-, pin distance=12pt]90:{\tiny{$1^3$}}}] {};
\node at (0.37,-0.794)[pin={[pin edge=<-, pin distance=12pt]270:{\tiny{$1^3$}}}] {};
%2
\node at (0.25,0.75)[pin={[pin edge=<-, pin distance=12pt]90:{\tiny{$2'$}}}] {};
\node at (0.02,-0.25)[pin={[pin edge=<-, pin distance=12pt]270:{\tiny{$2'$}}}] {};
%3
\node at (0.4, 0.79){\tiny{$3'$}};
\node at (0.63, -0.29) {\tiny{$3'$}};
%   4  
\node at (0.61, 0.79) {\tiny{$4'$}};
\node at (0.12, -0.12) {\tiny{$5'$}};
%   5 
\node at (0.76, 0.79) {\tiny{$5'$}};
\node at (0.27, -0.12) {\tiny{$4'$}};
%a
\draw[thick,->][green] (0.0214,0.75)--(0.0223,-0.25);
\node at (0.05, 0.5) {$a$}; 
%b
\draw[thick,->][blue] (0.1228,0.75)--(0.1321,-0.089);
\node at (0.1,0.6) {$b$};
%c 
\draw[thick,->, dashed][red] (0.26277,0.75)-- (0.2711,-0.089);
\node at (0.3,0.7) {$c$};
%d
\draw[thick,->][purple] (0.36638,0.75)--(0.3712,-0.794);
\node at (0.34,0.5) {$d$};
%e
\draw[thick,->][orange] (0.45243,0.75)--(0.46122,-0.794);
\node at (0.5,0.6) {$e$};
%f
\draw[thick,->,dashed][blue] (0.59116,0.75)--(0.6021,-0.25);
\node at (0.64,0.6) {$f$};
%g
\draw[thick,->] (0.80013,0.75)--(0.8142,-0.25);
\node at (0.77,0.4) {$g$};
%      A
\node at (0.4, -0.27) {\tiny{$A$}};
\node at (-0.02, 0.27) {\tiny{$A$}};
%      B
\node at (0.857, 0.7) {\tiny{$B$}};
\draw[thin] [green](0.420,0.296)--(0.460,0.296); 
\node at (0.44, 0.15) {\tiny{$B$}};
%      C
\node at (0.75,-0.15){\tiny{$C$}} {};
\draw[thin] [green](0.191,0.161)--(0.231,0.161); 
\node at (0.211, 0.08) {\tiny{$C$}};
%      D
\node at (-0.02, 0.7) {\tiny{$D$}};
\draw[thin] [green](0.420,0.456)--(0.38,0.456); 
\node at (0.4, 0.228) {\tiny{$D$}};
%      E
\node at (0.172, 0.175) {\tiny{$E$}};
\node at (0.44, 0.4) {\tiny{$E$}};
%      F
\node at (0.4, 0.59) {\tiny{$F$}};
\node at (0.79, 0.1) {\tiny{$F$}};
%      G
\node at (0.211, 0.25) {\tiny{$G$}};
\node at (0.75, 0.1) {\tiny{$G$}};
%     H 
\node at (0.857, 0.15) {\tiny{$H$}};
\node at (0.44,-0.3) {\tiny{$H$}};
%     I 
\node at (0.31, -0.3) {\tiny{$I$}};
\node at (0.48,-0.3) {\tiny{$I$}};
%     J
\node at (0.52,-0.15) {\tiny{$J$}};
\node at (0.06,-0.15) {\tiny{$J$}};
\end{tikzpicture}
}%end scalebox
&\phantom{higher}
&
\raisebox{4.7 cm}{
\scalebox{1}{
\noindent
\begin{tikzpicture}[x=5cm,y=5cm] 
%make outline
\draw   (0,0)--(0,1)--(\alp, 1)--(\alp, 0.839) --(0.839, 0.839)--(0.839, \alp)--(1, \alp) -- (1,0) --cycle; 
%mark blue singularity
 \foreach \x/\y in {0.191/0.352, 0.420/0.648, 0.771/0.191%
} { \node at (\x,\y) [blue]{$\bullet$}; } 
%mark red singularity
 \foreach \x/\y in {0.044/0, 0.339/0, 0.5/0, 0.272/1, 0.691/0.839, 0.920/0.544%
} { \node at (\x,\y) [red]{$\bullet$}; } 
% slits to blue singularities as thin lines
\draw[thin] (0.191,0)--(0.191, 0.352);        
\draw[thin] (0.420, 0)--(0.420, 0.648); 
\draw[thin] (0.771, 0)--(0.771, 0.191);
%a
\draw[thin,->-=0.5][green](0.2504,1) to (0.095,0.77);
\draw[thin,->-=0.5][green] (0.095,0.77) to  (0.0223,0);
\node at (0.04,0.5) {$a$};
%f
\draw[thin,->-=0.5,dashed][blue](0.374,1) to [bend left] (0.095,0.77);
\draw[thin,->-=0.72,dashed][blue](0.095,0.77) .. controls (0.691,0.65) .. (0.6021,0);
\node at (0.68,0.5) {$f$};
%d
\draw[thin,->-=0.5][purple](0.1819,1) to [bend right] (0.095,0.77);
\draw[thin,->-=0.5][purple] (0.095,0.77)to [bend left] (0.3712,0);
\draw[thin,->-=0.5][purple] (0.9515,0.544)--(0.97,0);
\node at (0.34,0.56) {$d$};
%g
\draw[thin,->-=0.5](0.042,1) to [bend right] (0.095,0.77);
\draw[thin,->-=0.5] (0.095,0.77) .. controls (0.839,0.7) ..  (0.8142,0);
\node at (0.78,0.4) {$g$};
%e
\draw[thin,->-=0.5][orange] (0.111,1) to [bend right] (0.095,0.77);
\draw[thin,->-=0.5][orange] (0.095,0.77) .. controls (0.54,0.68) .. (0.46122,0);
\draw[thin,->-=0.5][orange] (0.881,0.544)--(0.891,0);
\node at (0.53,0.4) {$e$};
%b
\draw[thin,->-=0.5][blue] (0.78,0.839) .. controls (0.42,0.75) .. (0.095,0.77);
\draw[thin,->-=0.5][blue] (0.095,0.77) to  (0.1321,0);
\node at (0.15,0.45) {$b$};
%c 
\draw[thin,->-=0.5, dashed][red] (0.623,0.839) .. controls (0.544,0.8) ..  (0.095,0.77);
\draw[thin,->-=0.5, dashed][red] (0.095,0.77) to [bend left] (0.2711,0);
\node at (0.23,0.5) {$c$};
%basepoint
\node at (0.0957439,0.771845) [black]{$\star$};
\end{tikzpicture}
}%end raisebox
}%end scalebox
\end{tabular}
\caption{Homology generators for a zippered rectangle surface punctured at its singularities are represented by curves appropriately respecting the natural partition of the zippered rectangle.   Left:  zippered rectangle is formed by cut-and-paste operations on the model of Figure~\ref{f:hubbRepn}.    Each curve is closed off by a segment along the top horizontal edge. Right:  homologous curves, shown on the original model, passing through $x_0$ of Lemma~\ref{l:orbit2}.}
\label{f:zippered}
\end{figure} 
%----------------------------------------------------------------------------------

%-----------------------
\begin{Lem}\label{l:genPiOneOnX}  On the Arnoux-Yoccoz surface $X$ let the closed curves $a, b, c, d, e, f, g$  passing through $x_0$  be as given Figure~\ref{f:zippered},  then the classes of $\{a, b, c, d, e, f, g\}$  generate the fundamental group  $\pi_1(X,x_0)$.   
\end{Lem}
%-----------------------

\begin{proof}  Let $x_1\in X$ be the point of the `red' singularity and let $\delta$ be the straight path from $x_1$ to $x_0$, as shown in the right side of Figure~\ref{f:superTri}.    We show that the classes of the conjugate by   $\delta$ of $\{a, b, c, d, e, f, g\}$ generate the fundamental group  $\pi_1(X,x_1)$.  Since such conjugations give isomorphisms of fundamental groups,  the result follows.   

Let the curves $t,u,v,w,x,y,z$ be given as in Figure~\ref{f:superTri}.     Since generators for the fundamental group are found on the 1-skeleton of $X$,   the classes of $\{v^{-1}t,uv,wu^{-1},y x, zt, yw^{-1},zx\}$ generate the fundamental group  $\pi_1(X,x_1)$.     The cyclic relation given by the gluing of the 2-skeleton to the 1-skeleton shows that we can delete any of one these and still have a generating set.

Treating each of $a, b, c, d, e, f, g$ by considering it placed on the  triangulation of $X$ given in Figure~\ref{f:superTri}, we find   \[
\begin{aligned}
\delta a \delta ^{-1}&=(t^{-1}v)(z x)^{-1}(wy^{-1}),\\
\delta  b \delta ^{-1}&=\delta a \delta ^{-1}(uv)^{-1},\\   
\delta  c \delta ^{-1}&=(uv)(wu^{-1})(yx)^{-1} (t^{-1}v)(uv)^{-1},\\
\delta  d \delta ^{-1}&=(uv)(wu^{-1})(yx)^{-1} \delta a \delta ^{-1},\\
\delta  e \delta ^{-1}&=(uv)(wu^{-1})(z x)^{-1}(wu^{-1})(z x)^{-1}(wy^{-1}),\\
\delta  f \delta ^{-1}&=(uv)(wu^{-1}) (z x)^{-1} (wu^{-1}),\\
\delta g \delta ^{-1}&=(uv)(wu^{-1}) (z x)^{-1} (wy^{-1}) (yx)(wu^{-1})^{-1}(uv)^{-1}.
\end{aligned}
\]

But, we can now algebraically solve and find 
\[
\begin{aligned}
        t^{-1}v&=\delta a \delta ^{-1} \,  (\delta e \delta ^{-1})^{-1}\, \delta f \delta ^{-1},\\
              uv&= (\delta b \delta ^{-1})^{-1} \,\delta a \delta ^{-1},\\
      wu^{-1}&=(uv)^{-1} \,\delta gda^{-2}\delta^{-1}\, (t^{-1}v),\\
      zx&=(wu^{-1})\,(\delta f \delta^{-1})^{-1}\, (uv)(wu^{-1}),\\
              yx&= (uv)^{-1}  (t^{-1}v)\,(\delta c \delta^{-1})^{-1}\, (uv)(wu^{-1}) \,,\\
      wy^{-1}&= (t^{-1}v)^{-1} (zx)\,\delta a \delta ^{-1}.
\end{aligned}
\]
Therefore the conjugates of the classes of $a,b,c,d,e,f,g$ also generates $\pi_1(X,x_1)$.  The result thus holds.
\end{proof}

%-----------------------------Triangulation and superimposed ----------------------------------------------------- 
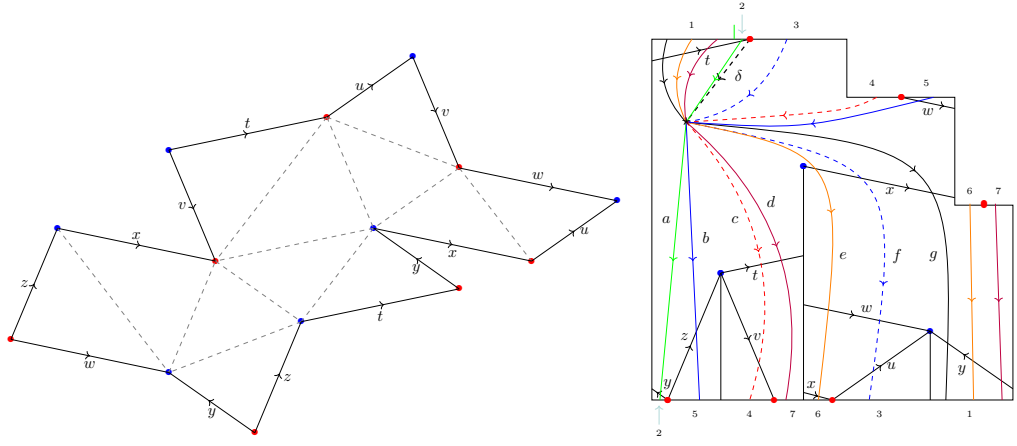
\begin{figure}[h]
\scalebox{0.6}{
\noindent
\begin{tabular}{lr} 
\begin{tikzpicture}[x=7cm,y=7cm] 
\coordinate (b1) at (0.191,0.352);
\coordinate (b2) at (0.42, 0.648);
\node at (b1) [blue]{$\bullet$};
\node at (b2) [blue]{$\bullet$};
\coordinate (b3) at (-0.228,0.896);
\node at (b3) [blue]{$\bullet$};
\coordinate  (b4) at (0.544,1.191);
\node at (b4) [blue]{$\bullet$};
\coordinate (b5) at (1.191,0.735);
\node at (b5) [blue]{$\bullet$};
\coordinate (b6) at (-0.228,0.191);
\node at (b6) [blue]{$\bullet$};
\coordinate (b7) at (-0.5805,0.648);
\node at (b7) [blue]{$\bullet$};
%
%mark red singularity
 %\foreach \x/\y in {0.044/0, 0.272/1, 0.691/0.839, 0.920/0.544%
%} { \node at (\x,\y) [red]{$\bullet$}; } 
\coordinate (r1) at (0.044,0);
\coordinate (r2) at (0.272,1);
\coordinate (r3) at (0.691,0.839);
\coordinate (r4) at (0.92,0.544);
\node at (r1) [red]{$\bullet$};
\node at (r2) [red]{$\bullet$};
\node at (r3) [red]{$\bullet$};
\node at (r4) [red]{$\bullet$};
\coordinate (r5) at (-0.0805,0.544);
\node at (r5) [red]{$\bullet$};
\coordinate (r6) at (-0.728,0.296);
\node at (r6) [red]{$\bullet$};
\coordinate (r7) at (0.691,0.456);
\node at (r7) [red]{$\bullet$};
\path (b2) edge[thin, gray, dashed] (r5);
\path (r5) edge[thin,-<-=0.5] node[left] {$v$} (b3);
\path (b3)  edge[thin,->-=0.5] node[above] {$t$} (r2);
\path (r2) edge[thin, gray, dashed] (r5);
\path (r2) edge[thin, gray, dashed] (b2);
\path (r2) edge[thin, gray, dashed] (r3);
\path (r3) edge[thin, gray, dashed] (b2);
\path (b4) edge[thin,->-=0.5] node[right] {$v$} (r3); 
\path (b4) edge[thin,-<-=0.5] node[left] {$u$} (r2);
\path (r3) edge[thin, gray, dashed] (r4);
\path (r3) edge[thin,->-=0.6] node[above] {$w$} (b5);
\path (r4) edge[thin,->-=0.5] node[right] {$u$} (b5);
\path (b2) edge[thin,->-=0.5] node[below] {$x$} (r4);
\path (b2) edge[thin, gray, dashed] (b1); 
\path (r5)   edge[thin, gray, dashed] (b1); 
\path (b1) edge[thin, gray, dashed] (b6); 
\path (b6) edge[thin, gray, dashed] (r5); 
\path (r5) edge[thin,-<-=0.5] node[above] {$x$} (b7);
\path (b7) edge[thin, gray, dashed] (b6); 
\path (r6) edge[thin,->-=0.5] node[left] {$z$} (b7);
\path (r6) edge[thin,->-=0.5] node[below] {$w$} (b6);
\path (b6) edge[thin,-<-=0.5] node[below] {$y$} (r1);
\path (b1) edge[thin,-<-=0.5] node[right] {$z$} (r1);
\path (r7) edge[thin,->-=0.5] node[below] {$y$} (b2);
\path (r7) edge[thin,-<-=0.5] node[below] {$t$} (b1);
\end{tikzpicture}
&
\begin{tikzpicture}[x=8cm,y=8cm] 
\draw   (0,0)--(0,1)--(\alp, 1)--(\alp, 0.839) --(0.839, 0.839)--(0.839, \alp)--(1, \alp) -- (1,0) --cycle; 
%mark blue singularity
 \foreach \x/\y in {0.191/0.352, 0.420/0.648, 0.771/0.191%
} { \node at (\x,\y) [blue]{$\bullet$}; } 
% slits to blue singularities as thin lines
\draw[thin] (0.191,0)--(0.191, 0.352);        
\draw[thin] (0.420, 0)--(0.420, 0.648); 
\draw[thin] (0.771, 0)--(0.771, 0.191); 
%           identifications  horizontal sides 
%  1
\draw[thin] [green](0.228,1)--(0.228,1.04); 
\node at (0.11, 1.04){\tiny{$1$}};
\node at (0.88, -0.04) {\tiny{$1$}};
%  2
\node at (0.25,1)[pin={[pin edge=<-, pin distance=12pt]90:{\tiny{$2$}}}] {};
\node at (0.02,0)[pin={[pin edge=<-, pin distance=12pt]270:{\tiny{$2$}}}] {};
%  3
\node at (0.4, 1.04){\tiny{$3$}};
\node at (0.63, -0.04) {\tiny{$3$}};
%   4  
\node at (0.61, 0.879) {\tiny{$4$}};
\node at (0.12, -0.04) {\tiny{$5$}};
%   5 
\node at (0.76, 0.879) {\tiny{$5$}};
\node at (0.27, -0.04) {\tiny{$4$}};
%   6 
\node at (0.88, 0.58) {\tiny{$6$}};
\node at (0.46, -0.04) {\tiny{$6$}};
%   7
\node at (0.96, 0.58) {\tiny{$7$}};
\node at (0.39, -0.04) {\tiny{$7$}};
\coordinate (b1) at (0.191,0.352);
\coordinate (b2) at (0.42, 0.648);
\node at (b1) [blue]{$\bullet$};
\node at (b2) [blue]{$\bullet$};
\coordinate (b4) at (0.771,0.191);
\coordinate (b7) at (-0.5805,0.648);
%intersection coordinates
\coordinate (t1) at (0,0.94);
\coordinate (t2) at (0.42,0.4);
\coordinate (w1) at (0.839,0.808);
\coordinate (w2) at (0.42,0.264);
\coordinate (x1) at (0.839,0.5608);
\coordinate (x2) at (0.42,0.0208);
\coordinate (y1) at (0,0.031);
\coordinate (y2) at (1,0.031);
\coordinate (r1) at (0.044,0);
\coordinate (r2) at (0.272,1);
\coordinate (r3) at (0.691,0.839);
\coordinate (r4) at (0.92,0.544);
\coordinate (r5) at (0.339,0);
\coordinate (r6) at (-0.728,0.296);
\coordinate (r8) at (0.5,0);
\path (r5) edge[thin,-<-=0.5] node[right] {$v$} (b1);
\path (t1) edge[thin,->-=0.5] node[right, below] {$\;\;\;t$} (r2);
\path (r3) edge[thin,->-=0.8] node[right, below] {$w$} (w1);
\path (w2) edge[thin, ->-=0.4] node[left, above] {$w$} (b4);
\path (r8) edge[thin,->-=0.5] node[right] {$u$} (b4);
\path (b2) edge[thin,->-=0.7] node[right, below] {$\phantom{mo}x$} (x1);
\path (x2) edge[thin, ->-=0.5] node[left, above] {$x_{\phantom{12}}$} (r8);
\path (y1) edge[thin,-<-=0.5] node[right, above] {$\;\;\;\,y$} (r1);
\path (b1) edge[thin,-<-=0.6] node[left] {$z$} (r1);
\path (y2) edge[thin,->-=0.6] node[left, below] {$y_{\phantom{123}}$} (b4); 
\path (t2) edge[thin,-<-=0.7] node[left, below] {$t_{\phantom{12}}$} (b1);
%Basepoint, in order 2 orbit 
\node at (0.0957439,0.771845) [black]{$\star$};
%delta, path from x_1 to x_0
\draw[->-=0.5, thick, dashed] (0.272,1)--(0.0957439,0.771845); 
\node at  (0.24,0.9) {$\delta$};
%%%pasting in curves
%a
\draw[thin,->-=0.5][green](0.2504,1) to (0.095,0.77);
\draw[thin,->-=0.5][green] (0.095,0.77) to  (0.0223,0);
\node at (0.04,0.5) {$a$};
%f
\draw[thin,->-=0.5,dashed][blue](0.374,1) to [bend left] (0.095,0.77);
\draw[thin,->-=0.72,dashed][blue](0.095,0.77) .. controls (0.691,0.65) .. (0.6021,0);
\node at (0.68,0.4) {$f$};
%d
\draw[thin,->-=0.5][purple](0.1819,1) to [bend right] (0.095,0.77);
\draw[thin,->-=0.5][purple] (0.095,0.77)to [bend left] (0.3712,0);
\draw[thin,->-=0.5][purple] (0.9515,0.544)--(0.97,0);
\node at (0.33,0.55) {$d$};
%g
\draw[thin,->-=0.5](0.042,1) to [bend right] (0.095,0.77);
\draw[thin,->-=0.5] (0.095,0.77) .. controls (0.839,0.7) ..  (0.8142,0);
\node at (0.78,0.4) {$g$};
%e
\draw[thin,->-=0.5][orange] (0.111,1) to [bend right] (0.095,0.77);
\draw[thin,->-=0.5][orange] (0.095,0.77) .. controls (0.54,0.68) .. (0.46122,0);
\draw[thin,->-=0.5][orange] (0.881,0.544)--(0.891,0);
\node at (0.53,0.4) {$e$};
%b
\draw[thin,->-=0.5][blue] (0.78,0.839) .. controls (0.42,0.75) .. (0.095,0.77);
\draw[thin,->-=0.5][blue] (0.095,0.77) to  (0.1321,0);
\node at (0.15,0.45) {$b$};
%c 
\draw[thin,->-=0.5, dashed][red] (0.623,0.839) .. controls (0.544,0.8) ..  (0.095,0.77);
\draw[thin,->-=0.5, dashed][red] (0.095,0.77) to [bend left] (0.2711,0);
\node at (0.23,0.5) {$c$};
%basepoint
\node at (0.0957439,0.771845) [black]{$\star$};
%putting red bullets last, so as to cover
\node at (r1) [red]{$\bullet$};
\node at (r2) [red]{$\bullet$};
\node at (r3) [red]{$\bullet$};
\node at (r4) [red]{$\bullet$};
%mark red singularity
 \foreach \x/\y in {0.044/0, 0.339/0, 0.5/0, 0.272/1, 0.691/0.839, 0.920/0.54%
} { \node at (\x,\y) [red]{$\bullet$}; } 
\end{tikzpicture}%
\end{tabular}
}
\caption{Left: a variant of Bowman's \cite{Bowman} triangulation of $X$.  Right: same, with dotted lines suppressed, superimposed on representation of $X$ given in Figure~\ref{f:hubbRepn};  the closed curves $a, \dots, g$ of Lemma~\ref{l:genPiOneOnX}; and, the path $\delta$ from $x_1$ to $x_0$ used in the proof of that lemma.}%
\label{f:superTri}%
\end{figure}
%-------------------------------------------------------------------------------------------

 \subsection{The blown-up Arnoux-Yoccoz surface, $\hat{X}$}   See \cite{BSW} for a detailed discussion of  blow-ups of translation surfaces.  The {\em blow-up of the Arnoux-Yoccoz surface} is the result of blowing up  its two singularities; this new surface $\hat{X}$ remains of genus three, but has two boundary components,  each a circle.   The {\em collapsing map},  $\mathfrak{c}:\hat{X}\to X$ is continuous and  is one-to-one off of these circles, while sending each of these to a singularity.   
  
 Fix $\hat{x}_0 \in \hat{X}$ to be the preimage of the $x_0$ under the collapsing map.   Since none of the closed curves $a, \dots, g$ passes through either singularity, we let $\hat{a}, \dots, \hat{g}$ be their respective preimages under the collapsing map.

%------------------------------- relation  ---------------------------------------------------
\begin{figure}[h]
\scalebox{1.1}{
\noindent
\begin{tikzpicture}[x=5cm,y=5cm] 
%make outline
\draw   (0,0)--(0,1)--(\alp, 1)--(\alp, 0.839) --(0.839, 0.839)--(0.839, \alp)--(1, \alp) -- (1,0) --cycle; 
%mark blue singularity
 \foreach \x/\y in {0.191/0.352, 0.420/0.648, 0.771/0.191%
} { \node at (\x,\y) [blue]{$\bullet$}; } 
%mark red singularity
 \foreach \x/\y in {0.044/0, 0.339/0, 0.5/0, 0.272/1, 0.691/0.839, 0.920/0.544%
} { \node at (\x,\y) [red]{$\bullet$}; } 
% slits to blue singularities as thin lines 
\draw[thin] (0.191,0)--(0.191, 0.352);        
\draw[thin] (0.420, 0)--(0.420, 0.648); 
\draw[thin] (0.771, 0)--(0.771, 0.191);
%           identifications  horizontal sides 
%  1
\draw[thin] [green](0.228,1)--(0.228,1.04); 
\node at (0.11, 1.04){\tiny{$1$}};
\node at (0.88, -0.04) {\tiny{$1$}};
%  2
\node at (0.25,1)[pin={[pin edge=<-, pin distance=12pt]90:{\tiny{$2$}}}] {};
\node at (0.02,0)[pin={[pin edge=<-, pin distance=12pt]270:{\tiny{$2$}}}] {};
%  3
\node at (0.4, 1.04){\tiny{$3$}};
\node at (0.63, -0.04) {\tiny{$3$}};
%   4  
\node at (0.61, 0.879) {\tiny{$4$}};
\node at (0.12, -0.04) {\tiny{$5$}};
%   5 
\node at (0.76, 0.879) {\tiny{$5$}};
\node at (0.27, -0.04) {\tiny{$4$}};
%   6 
\node at (0.88, 0.58) {\tiny{$6$}};
\node at (0.46, -0.04) {\tiny{$6$}};
%   7
\node at (0.96, 0.58) {\tiny{$7$}};
\node at (0.39, -0.04) {\tiny{$7$}};
%a
\draw[thin,->-=0.5][green](0.2504,1) to (0.095,0.77);
\draw[thin,->-=0.5][green] (0.095,0.77) to  (0.015,0);
\node at (0.04,0.5) {$a$};
%f
\draw[thin,->-=0.5,dashed][blue](0.374,1) to [bend left] (0.095,0.77);
\draw[thin,->-=0.72,dashed][blue](0.095,0.77) .. controls (0.691,0.65) .. (0.6021,0);
\node at (0.68,0.5) {$f$};
%d
\draw[thin,->-=0.5][purple](0.1819,1) to [bend right] (0.095,0.77);
\draw[thin,->-=0.5][purple] (0.095,0.77)to [bend left] (0.3712,0);
\draw[thin,->-=0.5][purple] (0.9515,0.544)--(0.97,0);
\node at (0.34,0.56) {$d$};
%g  Suppressed
%e
\draw[thin,->-=0.5][orange] (0.111,1) to [bend right] (0.095,0.77);
\draw[thin,->-=0.5][orange] (0.095,0.77) .. controls (0.54,0.68) .. (0.46122,0);
\draw[thin,->-=0.5][orange] (0.881,0.544)--(0.891,0);
\node at (0.53,0.4) {$e$};
%b
\draw[thin,->-=0.5][blue] (0.78,0.839) .. controls (0.42,0.7) .. (0.095,0.78);
\draw[thin,->-=0.5][blue] (0.095,0.77) to  (0.1321,0);
\node at (0.15,0.45) {$b$};
%c 
\draw[thin,->-=0.5, dashed][red] (0.623,0.839) .. controls (0.544,0.8) ..  (0.095,0.77);
\draw[thin,->-=0.5, dashed][red] (0.095,0.77) to [bend left] (0.2711,0);
\node at (0.23,0.5) {$c$};
%basepoint
\node at (0.0957439,0.771845) [black]{$\star$};
% partial filling (bottom)
\draw[opacity=0, fill=gray, fill opacity=0.3] (0.06,0) arc (0:170:0.02)--(0.047,0.2)--cycle;
\draw[opacity=0, fill=gray, fill opacity=0.3] (0.36,0) arc (0:170:0.02)--(0.35,0.2)--cycle;
\draw[opacity=0, fill=gray, fill opacity=0.3] (0.52,0) arc (0:170:0.02)--(0.5,0.2)--cycle;
% partial filling (top)
\draw[opacity=0, fill=gray, fill opacity=0.3] (0.25,0.98) arc (190:350:0.02)--(0.27,0.9)--cycle;
\draw[opacity=0, fill=gray, fill opacity=0.3] (0.67,0.83) arc (180:300:0.02)--(0.59,0.79)--cycle;
\draw[opacity=0, fill=gray, fill opacity=0.3] (0.9,0.53) arc (190:350:0.02)--(0.92,0.4)--cycle;
% partial filling (retracts)
\draw[opacity=0, fill=gray, fill opacity=0.3] (0.1,0.97) -- (0.1,0.94) --(0.12,0.94)--(0.12,0.97)--cycle;
\end{tikzpicture}
}%end scalebox
\caption{On $X$, a circle about $x_1$  is  homotopic relative to the singularities to $a  f^{-1}  e d^{-1} c b^{-1}$.   To see this, homotope the beginning portions of: $a$ and $b$ past the lower left gray triangle so as to collapse onto an arc (not pictured) about the left bottom representative of $x_1$;  $c$ and $d$ so as to collapse onto an arc  about the middle bottom representative of $x_1$; and,  $e$ and $f$ so as to collapse onto an arc  about the right bottom representative of $x_1$.  Also homotope
the ending portions of:  $a$ and $f$ to the top gray triangle and thereafter to collapse onto an arc about the left top  representative of $x_1$;  $e$ and $d$ past the gray square and thereafter pass the gray triangle to collapse onto an arc about the right top  representative of $x_1$;  and,  
of $c$ and $b$ so as to collapse onto an arc about the middle top representative.   %
It follows that   $\hat{a} \hat{f}^{-1} \hat{e} \, \hat{d}^{-1} \hat{c}\, \hat{b}^{-1}$ is freely homotopic on $\hat{X}$ to the blow up of $x_1$.  This is used in the proof of Lemma~\ref{l:genOnXhat}.}
\label{f:xxTryRelation}
\end{figure}
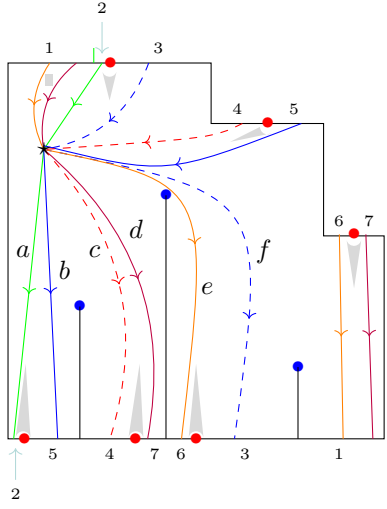 
%----------------------------------------------------------------------------------
 
%-----------------------
\begin{Lem}\label{l:genOnXhat}   The fundamental group $\pi_1(\hat{X}, \hat{x}_0)$ is freely generated by the classes of $\hat{a}, \hat{b}, \hat{c}, \hat{d}, \hat{e}, \hat{f}, \hat{g}$. 
\end{Lem}
%-----------------------
\begin{proof}  By Lemma~\ref{l:genPiOneOnX},  the set of classes of $\mathfrak c (\hat{a}), \dots, \mathfrak c(\hat{g})$ generates the fundamental group of $X$.         Figure~\ref{f:xxTryRelation} shows that 
$a  f^{-1}  e d^{-1} c b^{-1}$ is freely homotopic, relative to the singularities, on $X$  to a circle about $x_1$.     It follows that $\hat{a} \hat{f}^{-1} \hat{e} \, \hat{d}^{-1} \hat{c}\, \hat{b}^{-1}$ is freely homotopic on $\hat{X}$ to the blow up of $x_1$.     The result thus follows from Lemma~\ref{l:generateFreely}.
\end{proof}

 \section{Cokernel calculation}   By Lemma~\ref{l:orbit2},  each singularity of $X$  forms a full $h$-orbit.  Let $\hat{h}: \hat{X} \to\hat{X}$ be the `lift' of $h$ as per the discussion of Subsection~\ref{ss:blowUps} on $\hat{X}$.  Recall that $\hat{h}$ is also a pseudo-Anosov homeomorphism.  
Just as we did for our closed paths, also lift the path $\epsilon$ to $\hat{\epsilon}$ on $\hat{X}$, thus passing from $\hat{x}_0$ to $\hat{h}(\hat{x}_0)$. 

%-----------------------
\begin{Prop}\label{p:hatImages}   Up to basepoint-fixing homotopy,  the  map $\hat{h}$ acts  in the following manner: \begin{equation}\label{e:hatImages}
\begin{aligned}
\hat{a} &\mapsto  \hat{\epsilon}^{-1}\hat{f}\hat{c}\hat{\epsilon},\;\;  
\hat{b} \mapsto  \hat{\epsilon}^{-1}\hat{f} \hat{d}\hat{g}^{-1}\hat{\epsilon},\;\;
\hat{c}\mapsto  \hat{\epsilon}^{-1}\hat{f} \hat{e}\hat{g}^{-1}\hat{\epsilon}, \;\; 
\hat{d}  \mapsto  \hat{\epsilon}^{-1}\hat{f}^2 \hat{c}\hat{\epsilon},\\
\hat{e} &\mapsto  \hat{\epsilon}^{-1}\hat{g}\hat{a} \hat{c}\hat{\epsilon}\;\; 
\hat{f} \mapsto  \hat{\epsilon}^{-1}\hat{g}  \hat{b}\hat{\epsilon}, \;\;
\hat{g} \mapsto  \hat{\epsilon}^{-1}\hat{g}  \hat{c}\hat{\epsilon}.  
\end{aligned}
 \end{equation}  
 
 Furthermore,   $\hat{\epsilon}\, \hat{h}(\hat{\epsilon})$ is homotopic to $\hat{g}$. 
\end{Prop}
%-----------------------
\begin{proof}     One deduces from Figures~\ref{f:hahb} through \ref{f:hehfhg} that up to homotopy relative to the two singularities of $X$ and to $x_0$,    $h$ acts as $a \mapsto \epsilon^{-1} fc \epsilon, b\mapsto \epsilon^{-1} fdg^{-1} \epsilon, c\mapsto \epsilon^{-1} feg^{-1} \epsilon, d \mapsto \epsilon^{-1} f^2 c \epsilon, e\mapsto \epsilon^{-1} gac \epsilon, f \mapsto \epsilon^{-1} gb \epsilon, g \mapsto \epsilon^{-1} gc \epsilon$.     From Figures~\ref{f:ayAsAffDiffeo} and \ref{f:zippered}, one sees that $g$ and $ \epsilon h( \epsilon)$ define the same class in $\pi_1(X, x_0)$.    The result thus follows. 
\end{proof}

For later typographical ease,  let  
 $\mathcal B = \{A, B, C, D, E, F, G\}\subset H_1(\hat{X};\ZZ)$, where  $A =  [\hat{a}], B = [\hat{b}], \dots, G =[\hat{g}]$. 
  
%-----------------------
\begin{Lem}   The set $\mathcal B$   forms a 
$\mathbb Z$-basis of $H_1(\hat{X};\ZZ)$.   The $\mathbb Z$-linear map $\hat{h}_* - \text{Id}$ on $H_1(\hat{X};\ZZ)$, written additively, has as its matrix with respect to $\mathcal B$ the following: 
\[
\big(\, \hat{h}_*-\emph{Id}\,\big)_{\mathcal B} =
\begin{pmatrix}
-1& 0 & 0 & 0 & 1 & 0 & 0 \\
0 & -1& 0 & 0 & 0 & 1 & 0 \\
1 & 0 & -1& 1 & 1 & 0 & 1 \\
0 & 1 & 0 & -1& 0 & 0 & 0 \\
0 & 0 & 1 & 0 & -1& 0 & 0 \\
1 & 1 & 1 & 2 & 0 & -1& 0 \\
0 & -1&-1 & 0 & 1 & 1 &0
\end{pmatrix}.\]

\end{Lem}
%-----------------------

\begin{proof}  Since  the classes in $\pi_1(\hat{X}, \hat{x}_0)$ of $\hat{a}, \hat{b}, \hat{c}, \hat{d}, \hat{e}, \hat{f}, \hat{g}$ freely generate, their classes in homology form a basis.    From the previous proposition, it immediately follows that
  $\big(\, \hat{h}_*-\text{Id}\,\big)_{\mathcal B}$ is as displayed above. 
\end{proof} 
 
%-----------------------------------images of a and b ----------   
\begin{figure}[h]
\scalebox{1}{
%--------------------------------------------- 
\begin{tikzpicture}[x=5cm,y=5cm] 
%make outline
\draw   (0,0)--(0,1)--(\alp, 1)--(\alp, 0.839) --(0.839, 0.839)--(0.839, \alp)--(1, \alp) -- (1,0) --cycle; 
%mark blue singularity
 \foreach \x/\y in {0.191/0.352, 0.420/0.648, 0.771/0.191%
} { \node at (\x,\y) [blue]{$\bullet$}; } 
%mark red singularity
 \foreach \x/\y in {0.044/0, 0.339/0, 0.5/0, 0.272/1, 0.691/0.839, 0.920/0.544%
} { \node at (\x,\y) [red]{$\bullet$}; } 
% slits to blue singularities as thin lines 
\draw[thin] (0.191,0)--(0.191, 0.352);        
\draw[thin] (0.420, 0)--(0.420, 0.648); 
\draw[thin] (0.771, 0)--(0.771, 0.191); 
%
%h(a)
\draw[thick,->-=0.5][olive] (0.68,0.839)--(0.6,0.41);
\node at (0.69,0.6) {\tiny{$3$}};
\draw[thick,->-=0.5][olive] (0.6,0.41)--(0.6,0);
\node at (0.65,0.2) {\tiny{$1$}};
\draw[thick,->-=0.5][olive] (0.366,1)--(0.327,0);
\node at (0.3,0.5) {\tiny{$2$}};
%c 
\draw[thick,->-=0.5,dashed][red] (0.623,0.839)--(0.095,0.77);
\draw[thick,->-=0.5,dashed][red] (0.095,0.77) .. controls (0.191,0.5) ..   (0.2711,0);
\node at (0.1,0.6) {$c$};
%f
\draw[thick,->-=0.5,dashed][blue](0.33,1) to [bend left] (0.095,0.77);
\draw[thick,->-=0.72,dashed][blue](0.095,0.77) .. controls (0.691,0.7) and (0.55,0.65) .. (0.55,0);
\node at (0.2,0.9) {$f$};
%basepoint
\node at (0.0957439,0.771845) [black]{$\star$};
%image of Basepoint, in order 2 orbit 
\node at (0.595744,0.419643) [gray, fill opacity = 0.6]{$\star$};
%\epsilon
\draw[thin,->-=0.5, cyan] (0.095,0.77)--(0.595744, 0.77);
\draw[thin,->-=0.5, cyan] (0.595744, 0.77)--(0.595744, 0.44);
\node at (0.56,0.72) [cyan]{$\epsilon$};
 \end{tikzpicture}
\begin{tikzpicture}[x=5cm,y=5cm] 
%make outline
\draw   (0,0)--(0,1)--(\alp, 1)--(\alp, 0.839) --(0.839, 0.839)--(0.839, \alp)--(1, \alp) -- (1,0) --cycle; 
%mark blue singularity
 \foreach \x/\y in {0.191/0.352, 0.420/0.648, 0.771/0.191%
} { \node at (\x,\y) [blue]{$\bullet$}; } 
%mark red singularity
 \foreach \x/\y in {0.044/0, 0.339/0, 0.5/0, 0.272/1, 0.691/0.839, 0.920/0.544%
} { \node at (\x,\y) [red]{$\bullet$}; } 
% slits to blue singularities as thin lines 
\draw[thin] (0.191,0)--(0.191, 0.352);        
\draw[thin] (0.420, 0)--(0.420, 0.648); 
\draw[thin] (0.771, 0)--(0.771, 0.191); 
%h(b)
\draw[thin,->-=.3][blue,thick] (0.968,0.544) to [bend left] (0.5957,0.42);
\draw[thin,->-=.5][blue,thick] (0.5957,0.42)--(0.612,0);
\draw[thin,->-=.5][blue,thick] (0.383,1) -- (0.383,0);
%d  
\draw[thin,->-=0.5][purple](0.1819,1) to [bend right] (0.095,0.77);
\draw[thin,->-=0.5][purple] (0.095,0.77) .. controls (0.36,0.7) and (0.3712,0.65) .. (0.3712,0);
\draw[thin,->-=0.5][purple] (0.9515,0.544)--(0.95,0);
\node at (0.3,0.56) {$d$};
% -g
\draw[thin,-<-=0.5](0.042,1) to [bend right] (0.095,0.77);
\draw[thin,-<-=0.5] (0.095,0.77) .. controls (0.839,0.7) ..  (0.8142,0);
\node at (0.73,0.73) {$-g$};
%f
\draw[thick,->-=0.5,dashed][blue](0.33,1) to [bend left] (0.095,0.77);
\draw[thick,->-=0.72,dashed][blue](0.095,0.77) .. controls (0.691,0.7) and (0.55,0.65) .. (0.55,0);
\node at (0.2,0.9) {$f$};
%basepoint
\node at (0.0957439,0.771845) [black]{$\star$};
%image of Basepoint, in order 2 orbit 
\node at (0.595744,0.419643) [gray, fill opacity = 0.6]{$\star$};
%\epsilon
\draw[thin,->-=0.5, cyan] (0.095,0.77)--(0.595744, 0.77);
\draw[thin,->-=0.5, cyan] (0.595744, 0.77)--(0.595744, 0.44);
\node at (0.55,0.8) [cyan]{$\epsilon$};
% partial filling (retracts)
\draw[opacity=0, fill=gray, fill opacity=0.3] (0.08,0.97) -- (0.08,0.94) --(0.1,0.94)--(0.1,0.97)--cycle;
\end{tikzpicture}}
\caption{The images $h(a)$ shown in olive on the left --- with three portions enumerated ---  and $h(b)$ shown in blue on the right.   The class of $h(a)$ in $\pi_1(X, h(x_0))$ equals the class of  $\epsilon^{-1} fc \epsilon$, since we can homotope the curve formed by  $\epsilon^{-1}$ followed by the initial portion of $f$ to agree with  portion 1 of $h(a)$, the second portion of $f$ and the initial portion of $c$ to agree with portion 2 of $h(a)$ and finally the second portion of $c$ followed by $\epsilon$ to agree with portion 3 of $h(a)$.   Similarly, $h(b)$ is homotopic to $\epsilon^{-1} fdg^{-1} \epsilon$.}
\label{f:hahb}
\end{figure}
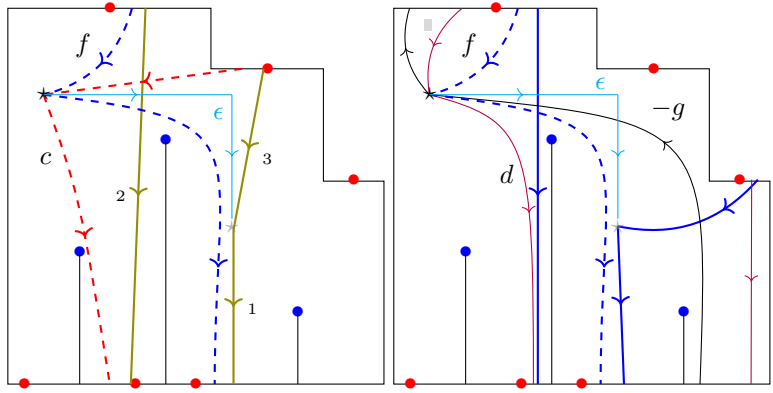
%------------------------------------- 

%-----------------------------------images of c and d ----------  
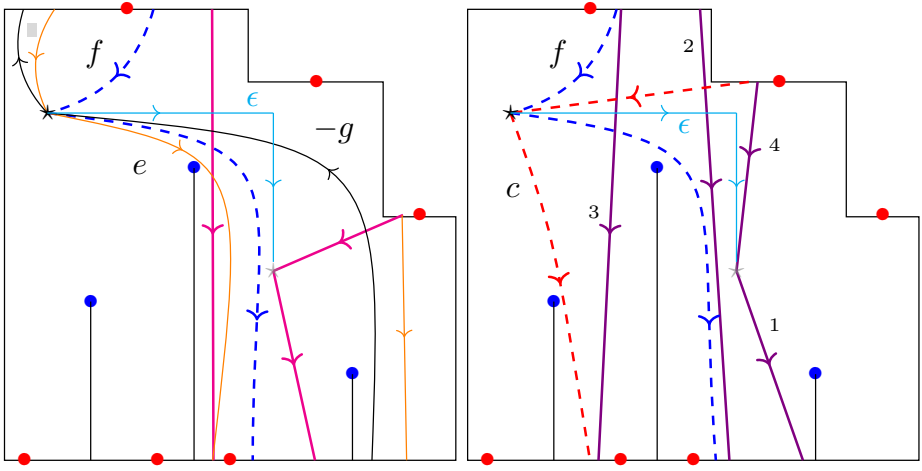
\begin{figure}[h]
\scalebox{1.2}{
%--------------------first h(c) -------------------------  
\begin{tikzpicture}[x=5cm,y=5cm] 
%make outline
\draw   (0,0)--(0,1)--(\alp, 1)--(\alp, 0.839) --(0.839, 0.839)--(0.839, \alp)--(1, \alp) -- (1,0) --cycle; 
%mark blue singularity
 \foreach \x/\y in {0.191/0.352, 0.420/0.648, 0.771/0.191%
} { \node at (\x,\y) [blue]{$\bullet$}; } 
%mark red singularity
 \foreach \x/\y in {0.044/0, 0.339/0, 0.5/0, 0.272/1, 0.691/0.839, 0.920/0.544%
} { \node at (\x,\y) [red]{$\bullet$}; } 
% slits to blue singularities as thin lines
\draw[thin] (0.191,0)--(0.191, 0.352);        
\draw[thin] (0.420, 0)--(0.420, 0.648); 
\draw[thin] (0.771, 0)--(0.771, 0.191); 
%h(c)
\draw[thick,->-=0.5][magenta] (0.883,0.544)--(0.595744,0.419643);
\draw[thick,->-=0.5][magenta] (0.595744,0.419643)--(0.688,0);
\draw[thick,->-=0.5][magenta] (0.46,1)--(0.463,0);
%f
\draw[thick,->-=0.5,dashed][blue](0.33,1) to [bend left] (0.095,0.77);
\draw[thick,->-=0.72,dashed][blue](0.095,0.77) .. controls (0.691,0.7) and (0.55,0.65) .. (0.55,0);
\node at (0.2,0.9) {$f$};
%e
\draw[thin,->-=0.5][orange] (0.111,1) to [bend right] (0.095,0.77);
\draw[thin,->-=0.3][orange] (0.095,0.77) .. controls (0.54,0.68) .. (0.46122,0);
\draw[thin,->-=0.5][orange] (0.881,0.544)--(0.891,0);
\node at (0.3,0.65) {$e$};
% -g
\draw[thin,-<-=0.5](0.042,1) to [bend right] (0.095,0.77);
\draw[thin,-<-=0.5] (0.095,0.77) .. controls (0.839,0.7) ..  (0.8142,0);
\node at (0.73,0.73) {$-g$};
%basepoint
\node at (0.0957439,0.771845) [black]{$\star$};
%image of Basepoint, in order 2 orbit 
\node at (0.595744,0.419643) [gray, fill opacity = 0.6]{$\star$};
%\epsilon
\draw[thin,->-=0.5, cyan] (0.095,0.77)--(0.595744, 0.77);
\draw[thin,->-=0.5, cyan] (0.595744, 0.77)--(0.595744, 0.44);
\node at (0.55,0.8) [cyan]{$\epsilon$};
% partial filling (retracts)
\draw[opacity=0, fill=gray, fill opacity=0.3] (0.05,0.97) -- (0.05,0.94) --(0.07,0.94)--(0.07,0.97)--cycle;
\end{tikzpicture}
%
%  and, now, h(d) 
%
\begin{tikzpicture}[x=5cm,y=5cm] 
%make outline
\draw   (0,0)--(0,1)--(\alp, 1)--(\alp, 0.839) --(0.839, 0.839)--(0.839, \alp)--(1, \alp) -- (1,0) --cycle; 
%mark blue singularity
 \foreach \x/\y in {0.191/0.352, 0.420/0.648, 0.771/0.191%
} { \node at (\x,\y) [blue]{$\bullet$}; } 
%mark red singularity
 \foreach \x/\y in {0.044/0, 0.339/0, 0.5/0, 0.272/1, 0.691/0.839, 0.920/0.544%
} { \node at (\x,\y) [red]{$\bullet$}; } 
% slits to blue singularities as thin lines 
\draw[thin] (0.191,0)--(0.191, 0.352);        
\draw[thin] (0.420, 0)--(0.420, 0.648); 
\draw[thin] (0.771, 0)--(0.771, 0.191); 
%h(d)
\draw[thick,->-=0.4][violet] (0.6425,0.839)--(0.595744,0.419643);
\node at (0.68,0.7) {\tiny{$4$}};
\draw[thick,->-=0.5][violet] (0.595744,0.419643)--(0.743,0);
\node at (0.68,0.3) {\tiny{$1$}};
\draw[thick,->-=0.4][violet] (0.515,1)--(0.58,0);
\node at (0.49,0.92) {\tiny{$2$}};
\draw[thick,->-=0.5][violet] (0.34,1)--(0.29,0);
\node at (0.28,0.55) {\tiny{$3$}};
%c 
\draw[thick,->-=0.5,dashed][red] (0.623,0.839)--(0.095,0.77);
\draw[thick,->-=0.5,dashed][red] (0.095,0.77) .. controls (0.191,0.5) ..   (0.2711,0);
\node at (0.1,0.6) {$c$};
%f
\draw[thick,->-=0.5,dashed][blue](0.33,1) to [bend left] (0.095,0.77);
\draw[thick,->-=0.72,dashed][blue](0.095,0.77) .. controls (0.63,0.7) and (0.5,0.65) .. (0.55,0);
\node at (0.2,0.9) {$f$};
%basepoint
\node at (0.0957439,0.771845) [black]{$\star$};
%image of Basepoint, in order 2 orbit 
\node at (0.595744,0.419643) [gray, fill opacity = 0.6]{$\star$};
%\epsilon
\draw[thin,->-=0.7, cyan] (0.095,0.77)--(0.595744, 0.77);
\draw[thin,->-=0.5, cyan] (0.595744, 0.77)--(0.595744, 0.44);
\node at (0.48,0.74) [cyan]{$\epsilon$};
\end{tikzpicture}}
\caption{The images $h(c)$ shown in red on the left,  and $h(d)$ shown in violet on the right.   The curve $\epsilon^{-1} feg^{-1} \epsilon$ homotopes to $h(c)$.     The curve $\epsilon^{-1} f^2 c \epsilon$ homotopes to $h(d)$.  For this latter, homotope the curve formed by  $\epsilon^{-1}$ followed by the initial portion of $f$ to agree with  portion 1 of $h(d)$; the second portion of $f$ and the initial portion of the second copy of $f$ to agree with portion 2 of $h(d)$; the second portion of $f$ and the initial portion of $c$ to agree with portion 3 of $h(d)$;  and the second portion of $c$ followed by $\epsilon$ to agree with portion 4 of $h(d)$.  
}
\label{f:hchd}
\end{figure}
%-------------------------------------------  

%---------------images of e,f,g ---------------------------- 
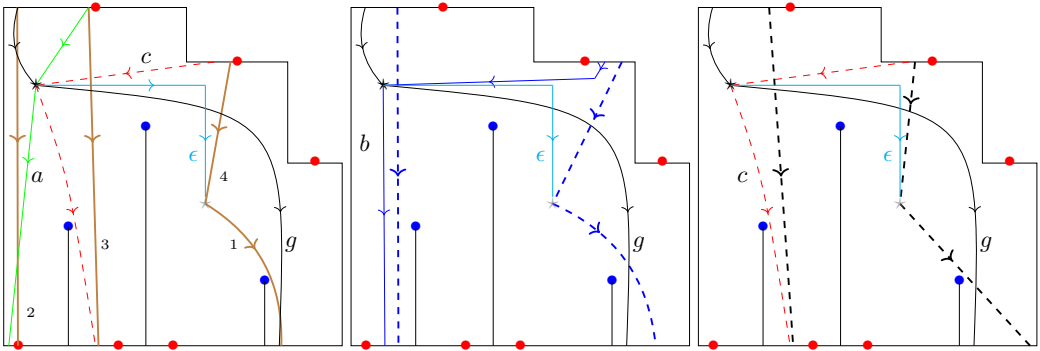
\begin{figure}[h]
\scalebox{.9}{
%-----------------------first h(e) ---------------------- 
\begin{tikzpicture}[x=5cm,y=5cm] 
%make outline
\draw   (0,0)--(0,1)--(\alp, 1)--(\alp, 0.839) --(0.839, 0.839)--(0.839, \alp)--(1, \alp) -- (1,0) --cycle; 
%mark blue singularity
 \foreach \x/\y in {0.191/0.352, 0.420/0.648, 0.771/0.191%
} { \node at (\x,\y) [blue]{$\bullet$}; } 
%mark red singularity
 \foreach \x/\y in {0.044/0, 0.339/0, 0.5/0, 0.272/1, 0.691/0.839, 0.920/0.544%
} { \node at (\x,\y) [red]{$\bullet$}; } 
% slits to blue singularities as thin lines
\draw[thin] (0.191,0)--(0.191, 0.352);        
\draw[thin] (0.420, 0)--(0.420, 0.648); 
\draw[thin] (0.771, 0)--(0.771, 0.191); 
%h(e)
\draw[thick,->-=0.5][brown] (0.67,0.839)--(0.595744,0.419643);
\node at (0.65,0.5) {\tiny{$4$}};
\draw[thick,->-=0.4][brown] (0.595744,0.419643)  to [bend left] (0.82,0);
\node at (0.68,0.3) {\tiny{$1$}};
\draw[thick,->-=0.4][brown] (0.251,1)--(0.28,0);
\node at (0.3,0.3) {\tiny{$3$}};
\draw[thick,->-=0.4][brown] (0.04,1)--(0.042,0);
\node at (0.08,0.1) {\tiny{$2$}};
% g
\draw[thin,->-=0.5](0.042,1) to [bend right] (0.095,0.77);
\draw[thin,->-=0.7] (0.095,0.77) .. controls (0.839,0.7) ..  (0.8142,0);
\node at (0.85,0.3) {$g$};
%a
\draw[thin,->-=0.5][green](0.2504,1) to (0.095,0.77);
\draw[thin,->-=0.3][green] (0.095,0.77) to  (0.015,0);
\node at (0.1,0.5) {$a$};
%c 
\draw[thin,->-=0.5,dashed][red] (0.623,0.839)--(0.095,0.77);
\draw[thin,->-=0.5,dashed][red] (0.095,0.77) .. controls (0.191,0.5) ..   (0.2711,0);
\node at (0.42,0.85) {$c$};
%basepoint
\node at (0.0957439,0.771845) [black]{$\star$};
%image of Basepoint, in order 2 orbit 
\node at (0.595744,0.419643) [gray, fill opacity = 0.6]{$\star$};
%\epsilon
\draw[thin,->-=0.7, cyan] (0.095,0.77)--(0.595744, 0.77);
\draw[thin,->-=0.5, cyan] (0.595744, 0.77)--(0.595744, 0.44);
\node at (0.56,0.56) [cyan]{$\epsilon$};
\end{tikzpicture}
%
% and, now h(f)
%
\begin{tikzpicture}[x=5cm,y=5cm] 
%make outline
\draw   (0,0)--(0,1)--(\alp, 1)--(\alp, 0.839) --(0.839, 0.839)--(0.839, \alp)--(1, \alp) -- (1,0) --cycle; 
%mark blue singularity
 \foreach \x/\y in {0.191/0.352, 0.420/0.648, 0.771/0.191%
} { \node at (\x,\y) [blue]{$\bullet$}; } 
%mark red singularity
 \foreach \x/\y in {0.044/0, 0.339/0, 0.5/0, 0.272/1, 0.691/0.839, 0.920/0.544%
} { \node at (\x,\y) [red]{$\bullet$}; } 
% slits to blue singularities as thin lines 
\draw[thin] (0.191,0)--(0.191, 0.352);        
\draw[thin] (0.420, 0)--(0.420, 0.648); 
\draw[thin] (0.771, 0)--(0.771, 0.191);
%h(f)
\draw[thick,->-=0.4,dashed][blue](0.8,0.839)--(0.595744,0.419643);
\draw[thick,->-=0.3,dashed][blue](0.595744,0.419643) to [bend left] (0.9,0);
\draw[thick,->-=0.5,dashed][blue](0.137,1)--(0.14,0);
% g
\draw[thin,->-=0.5](0.042,1) to [bend right] (0.095,0.77);
\draw[thin,->-=0.7] (0.095,0.77) .. controls (0.839,0.7) ..  (0.8142,0);
\node at (0.85,0.3) {$g$};
% b
\draw[thin,->-=0.5][blue] (0.75,0.839)--(0.72,0.79019);
\draw[thin,->-=0.5][blue] (0.72,0.79019)--(0.095,0.77019);
\draw[thin,->-=0.5][blue] (0.095,0.77019)--(0.1,0);
\node at (0.04,0.6) {$b$};
%basepoint
\node at (0.0957439,0.771845) [black]{$\star$};
%image of Basepoint, in order 2 orbit 
\node at (0.595744,0.419643) [gray, fill opacity = 0.6]{$\star$};
%\epsilon
\draw[thin, cyan] (0.095,0.77)--(0.595744, 0.77);
\draw[thin,->-=0.5, cyan] (0.595744, 0.77)--(0.595744, 0.44);
\node at (0.56,0.56) [cyan]{$\epsilon$};
\end{tikzpicture}
%
% and finally h(g)
%
\begin{tikzpicture}[x=5cm,y=5cm] 
%make outline
\draw   (0,0)--(0,1)--(\alp, 1)--(\alp, 0.839) --(0.839, 0.839)--(0.839, \alp)--(1, \alp) -- (1,0) --cycle; 
%mark blue singularity
 \foreach \x/\y in {0.191/0.352, 0.420/0.648, 0.771/0.191%
} { \node at (\x,\y) [blue]{$\bullet$}; } 
%mark red singularity
 \foreach \x/\y in {0.044/0, 0.339/0, 0.5/0, 0.272/1, 0.691/0.839, 0.920/0.544%
} { \node at (\x,\y) [red]{$\bullet$}; } 
% slits to blue singularities as thin lines 
\draw[thin] (0.191,0)--(0.191, 0.352);        
\draw[thin] (0.420, 0)--(0.420, 0.648); 
\draw[thin] (0.771, 0)--(0.771, 0.191);
%h(g)
\draw[thick,->-=0.3,dashed](0.64,0.839)--(0.595744,0.419643);
\draw[thick,->-=0.4,dashed](0.595744,0.419643)--(0.98,0);
\draw[thick,->-=0.5,dashed] (0.209,1)--(0.28,0);
% g
\draw[thin,->-=0.5](0.042,1) to [bend right] (0.095,0.77);
\draw[thin,->-=0.7] (0.095,0.77) .. controls (0.839,0.7) ..  (0.8142,0);
\node at (0.85,0.3) {$g$};
%c 
\draw[thin,->-=0.5,dashed][red] (0.623,0.839)--(0.095,0.77);
\draw[thin,->-=0.5,dashed][red] (0.095,0.77) .. controls (0.191,0.5) ..   (0.2711,0);
\node at (0.13,0.5) {$c$};
%basepoint
\node at (0.0957439,0.771845) [black]{$\star$};
%image of Basepoint, in order 2 orbit 
\node at (0.595744,0.419643) [gray, fill opacity = 0.6]{$\star$};
%\epsilon
\draw[thin, cyan] (0.095,0.77)--(0.595744, 0.77);
\draw[thin,->-=0.5, cyan] (0.595744, 0.77)--(0.595744, 0.44);
\node at (0.56,0.56) [cyan]{$\epsilon$};
\end{tikzpicture}}
\caption{Left to right: the curves $\hat{h}_*(e)$ in brown, $\hat{h}_*(f)$ in dashed blue, and $\hat{h}_*(g)$ in dashed black.  One finds that $ \epsilon^{-1} gac \epsilon$ homotopes to $h(e)$;   $\epsilon^{-1} gb \epsilon, g \mapsto \epsilon^{-1} gc \epsilon$ to $h(f)$; and $\epsilon^{-1} gc \epsilon$ to $h(g)$.}
\label{f:hehfhg}
\end{figure}
%-----------------------

Calculating the Smith Normal Form of $\big(\, \hat{h}_*-\text{Id}\,\big)_{\mathcal B}$ allows one to find the following.
%-----------------------
\begin{Lem}\label{l:cokerIdentified}   The cokernel of $\hat{h}_* - \text{Id} :  H_1(\hat{X};\ZZ) \to H_1(\hat{X};\ZZ)$ is such that
\[ \emph{Coker}(\hat{h}_* - \emph{Id})  \cong \ZZ/2\ZZ \oplus \ZZ.\]
Furthermore,   let
$\mathcal B' =  \{-A+ C +F, -B+D+F-G,-C+E+F-G, C-D+2 F, A+C-E+G,F, G\}$; 
this is  an ordered $\mathbb Z$-basis of $H_1(\hat{X};\ZZ)$ such that the congruence class modulo $\emph{im}(\hat{h}_* - \text{Id})$ of the final element $G= [\hat{g}]$  generates the free subgroup of the cokernel and the previous element, $F$ similarly gives a generator of the torsion subgroup.  Each of the remaining elements of $\mathcal B'$ give the zero element of the cokernel. 
\end{Lem}
%-----------------------

\section{Multivariable Lefschetz function of the mapping torus $M_{\hat{h}}$} 

  Relying on Subsection~\ref{ss:toFriedsTheorem},  the main step of the current section is to compute the action of an elevation $h_Q$ of (a homotopically equivalent map to) $\hat{h}$ to a cover whose deck transformation group is $Q = Q_{\hat{h}}$ on cells of an appropriate CW-decomposition.  In Subsection~\ref{ss:homCovElevation} we determine how an elevation to the universal abelian cover of $\hat{X}$ acts on appropriate lifts of paths,  then use this in Subsection~\ref{ss:QcoverElevation} to determine how $h_Q$ similarly acts.   The action on the cells of the CW-decomposition follows, as shown in Subsection~\ref{ss:actOnCells}.  Using Fried's Theorem, Theorem~\ref{t:frieds}, the multivariable Lefschetz function is then given in Subsection~\ref{theFun}.   
The main result is implied by the results of the ensuing Subsection~\ref{ss:notHere}.  The section ends with further results and discussion.

\subsection{Action on homology cover of $\hat{X}$}\label{ss:homCovElevation}
Let $p_{\mathcal A}:\hat{X}_{\mathcal A}\to \hat{X}$ of $\hat{X}$ be the  universal abelian cover of $\hat{X}$;  this is the regular covering space of $\hat{X}$ whose deck transformation group is isomorphic to  $H_1(\hat{X};\ZZ)$.  
      Choose a basepoint $x_{\mathcal A}\in \hat{X}_{\mathcal A}$ above the basepoint $\hat{x}_0$ of $\hat{X}$.    
  Let 
$\tilde{a}$  denote the lift of $\hat{a}$ going from   $x_{\mathcal A}$ to $A  x_{\mathcal A}$, and similarly for $\tilde{b}, \dots, \tilde{g}$,  where here and below, we use juxtaposition to represent the action of the deck transformation group $H_1(\hat{X};\ZZ)$ on $\hat{X}_{\mathcal A}$.    Thus, for $J\in H_1(\hat{X},\ZZ)$ and $k\in \{a,b,c,d,e,f,g\}$,   we have that $J\tilde{k}$ is the lift of $\hat{k}$ with initial point $J x_{\mathcal A}$.     By definition, $\tilde{k}^{-1}$ is the same set of points as $\tilde{k}$, but the curve is taken in the opposite direction: it goes from $K  x_{\mathcal A}$ to $x_{\mathcal A}$.  Due to this, the lift of $\hat{k}^{-1}$ with initial point $J x_{\mathcal A}$ is in fact $J K^{-1} \tilde{k}^{-1}$, whose terminal endpoint is $J K^{-1}x_{\mathcal A}$.

Since $(\hat{h}\circ p_{\mathcal A})_*(\pi_1(\hat{X}_{\mathcal A})\,) \subset (p_{\mathcal A})_*(\pi_1(\hat{X}_{\mathcal A})\,)$, there is  an elevation $\tilde{h}:\hat{X}_{\mathcal A}\to \hat{X}_{\mathcal A}$  of $\hat{h}$.   Let $\tilde{\epsilon}$ be the path passing from $x_{\mathcal A}$ to $\tilde{h}(x_{\mathcal A})$ that lifts $\hat{\epsilon}$.   
Recall  that we take the multiplicative version of our homology groups.   We ``factor" curves from left to right, and use single dots to separate notation for curves.    Our lifts must be of course be connected curves, and thus for example a lift of a composite curve  $\hat{k}\hat{\ell}$ beginning at   $\hat{x}_0$ has the form $\tilde{k}\cdot K\tilde{\ell}$.

Thus,  \eqref{e:hatImages}  yields (using at times the fact that our group is abelian) the following. 

%-----------------------
\begin{Lem}\label{l:tildeIms}   The map $\tilde{h}$ acts as 
\begin{equation}\label{e:tildeImages}
\begin{aligned}
\tilde{a} &\mapsto \tilde{\epsilon}^{-1} \cdot \tilde{f} \cdot F \tilde{c} \cdot FC \tilde{\epsilon},\\
\tilde{b} &\mapsto  \tilde{\epsilon}^{-1} \cdot \tilde{f} \cdot F \tilde{d} \cdot FDG^{-1} \tilde{g}^{-1}\cdot FDG^{-1}\tilde{\epsilon},\\
\tilde{c} &\mapsto  \tilde{\epsilon}^{-1} \cdot \tilde{f} \cdot F \tilde{e} \cdot FEG^{-1} \tilde{g}^{-1}\cdot FEG^{-1}\tilde{\epsilon},\\
\tilde{d} &\mapsto  \tilde{\epsilon}^{-1} \cdot \tilde{f} \cdot F \tilde{f} \cdot F^2 \tilde{c} \cdot F^2C \tilde{\epsilon},\\
\tilde{e} &\mapsto  \tilde{\epsilon}^{-1} \cdot \tilde{g} \cdot{G} \tilde{a} \cdot GA \tilde{c} \cdot GAC \tilde{\epsilon},\\
\tilde{f} &\mapsto  \tilde{\epsilon}^{-1} \cdot \tilde{g} \cdot{G} \tilde{b} \cdot GB \tilde{\epsilon},\\
\tilde{g}  &\mapsto  \tilde{\epsilon}^{-1} \cdot \tilde{g} \cdot{G} \tilde{c} \cdot GC \tilde{\epsilon},\\
\tilde{\epsilon} &\mapsto  \tilde{\epsilon}^{-1} \cdot \tilde{g}.
\end{aligned}
\end{equation}
\end{Lem}
%-------------------

\subsection{Elevated action  commuting with deck group  on maximal free abelian cover of $\hat{X}$}\label{ss:QcoverElevation}

In light of Lemma~\ref{l:cokerIdentified}, we let 
   $Q$ be the torsion free quotient of  $\text{Coker}(\hat{h}_* - \text{Id})$  and $\psi_Q: H_1(\hat{X}; \ZZ)\to Q$ be the natural quotient homomorphism. From that lemma, $Q=<\psi_Q(G)>$; let us denote this cyclic generator by $\nu$. 
It is fairly straightforward to show that $Q$ is the largest abelian torsion-free quotient of $\pi_1(\hat{X},x_0)$ 
on which $\hat{h}_{*}$ acts trivially. Let  $p_Q:\hat{X}_Q\to \hat{X}$ be the regular covering space  whose deck transformation group is isomorphic to $Q$, and let $x_Q$ be its basepoint.   There is thus a  covering transformation $r:\hat{X}_{\mathcal A}\to X_Q$ such that $p_{\mathcal A}=p_Q\circ r$.   Define $a_Q = r(\tilde{a}),b_Q = r(\tilde{b}), \dots, g_Q = r(\tilde{g})$ and $\epsilon_Q = r(\tilde{\epsilon})$.    Note that for $K\in H_1(\hat{X};\ZZ)$ and $\kappa$ a path in $\hat{X}_{\mathcal A}$ beginning at  $x_{\mathcal A}$, we have   $\psi_Q(K)r(\kappa) = r(K\kappa)$.\\
 
We can and do choose our elevation $h_Q$ of $\hat{h}$ such that $\tilde{h}$ is also an elevation of it.
With this, $r\circ \tilde{h}=h_Q\circ r$.    We now determine the action of $h_Q$ by using $\psi_Q, r$ and the action of $\tilde{h}$.
 
%-----------------------
\begin{Lem}\label{l:qIms}   The  map $h_{Q}$ acts as 

\begin{equation}\label{e:qImages}
 \begin{aligned}
a_Q&\mapsto \epsilon_{Q}^{-1}\cdot f_Q \cdot  c_Q \cdot \epsilon_Q,\\
b_Q&\mapsto \epsilon_{Q}^{-1}\cdot f_Q\cdot  d_Q  \cdot   \nu^{-1} g_{Q}^{-1} \cdot  \nu^{-1}  \epsilon_Q, \\
c_Q&\mapsto \epsilon_{Q}^{-1}\cdot f_Q\cdot e_Q \cdot  g_{Q}^{-1} \cdot  \epsilon_Q,\\
d_Q&\mapsto \epsilon_{Q}^{-1}\cdot f_Q\cdot f_Q \cdot c_Q \cdot \epsilon_Q,\\
e_Q&\mapsto \epsilon_{Q}^{-1}\cdot g_Q \cdot \nu a_Q \cdot \nu c_Q\cdot \nu \epsilon_Q,\\
f_Q&\mapsto \epsilon_{Q}^{-1}\cdot g_Q \cdot \nu  b_Q \cdot \epsilon_Q,\\
g_Q&\mapsto \epsilon_{Q}^{-1}\cdot g_Q \cdot  \nu c_Q \cdot \nu \epsilon_Q,\\ 
\epsilon_Q &\mapsto \epsilon_{Q}^{-1}\cdot g_Q\,.
\end{aligned}
\end{equation}

\end{Lem}
%---------------

\begin{proof}
Recall that $\psi_Q(G) = \nu$.    Consider  the ordered $\mathbb Z$-basis  of the  deck transformation group  $H_1(\hat{X};\ZZ)$ from Lemma~\ref{l:cokerIdentified}, written multiplicatively:
\[\mathcal B' =  (A^{-1}CF, B^{-1}DFG, C^{-1}EFG^{-1}, CD^{-1}F^2, ACE^{-1}G, F, G).\]   By that lemma,  the first six entries of $\mathcal B'$ generate the kernel of $\psi_Q$.  In particular, $F$ is in the kernel. Multiplying the third and fifth of these entries, we find $A \in \ker \psi_Q$.   The first entry shows that  $C \in \ker \psi_Q$.  The fourth,  that $D$ is also.  The second entry shows that  $\psi_Q(B) = \nu$,  and the fifth that $\psi_Q(E) = \nu$.  One can now evaluate the homomorphism $\psi_Q$  on the elements that appear in the image values above to find that they are those given in \eqref{e:qImages}.
\end{proof}

\subsection{Action of  cellular version of $h_Q$}\label{ss:actOnCells}   By Fried's Theorem, Theorem~\ref{t:frieds},  we can replace  $\hat{h}:\hat{X}\to \hat{X}$ with a cellular map on a one dimensional CW complex  and similarly replace the elevations $\tilde{h}$ and $h_{Q}$.    For simplicity's sake,   we retain the same notation for the spaces, maps, curves on the spaces, homology classes, etc.     We can homotope so as to collapse $\epsilon$, and thus our new $\hat{X}$ is the wedge of seven circles, with map $\hat{h}$ preserving their common point.  We cellulate in the simplest manner:   there are seven 1-cells, given by the (now open) curves $\hat{a}, \dots, \hat{g}$ and a single 0-cell.    Correspondingly,  $x_Q$ generates the $\ZZ[Q]$-module of 0-cells of $X_Q$, and $a_Q, \dots, g_Q$ generate the $\ZZ[Q]$-module of 1-cells of $X_Q$.   

%-----------------------
\begin{Lem}\label{l:qCellAction} The homotopically equivalent  cellular version of $h_Q$ acts on the $\ZZ[Q]$-module of 1-cells of $X_Q$, with respect to the ordered $\ZZ[Q]$-basis $\mathcal C_1 = (a_Q, \dots, g_Q)$, as
\[ 
\mathcal F_1 = \begin{pmatrix} 
0&0              &0   &0&\nu &0        &0\\
0&0              &0   &0&0 &\nu        &0\\ 
1&0              &0  &1&\nu &0        &\nu\\
0&1              &0  &0&0 &0        &0\\
0&0              &1  &0&0&0         &0\\
1&1              &1  &2&0        &0         &0\\ 
0&-\nu^{-1}   &-1&0&1       &1        &1
\end{pmatrix}.
\] 
Moreover, $h_Q$ acts trivially on the 0-cells.
\end{Lem}
%-----------------------

\begin{proof}
From Lemma~\ref{l:qIms}, since all appearances of $\epsilon^{\pm 1}$ become  trivial in our cellular model,  we find that $h_Q$ acts on the additive group of 1-cells as 
\[
 \begin{aligned}
a_Q&\mapsto c_Q +  f_Q,\\
b_Q&\mapsto  d_Q+ f_Q  - \nu^{-1} g_{Q}, \\
c_Q&\mapsto e_Q+  f_Q  - g_{Q},\\
d_Q&\mapsto  c_Q+ 2 f_Q,\\
e_Q&\mapsto  \nu a_Q + \nu c_Q + g_Q,\\
f_Q&\mapsto  \nu b_Q + g_Q,\\
g_Q&\mapsto \nu c_Q+ g_Q. 
\end{aligned}
\] 
\end{proof} 

%----------------
\begin{Rmk}  In fact, there is no actual need to collapse $\epsilon$.    With the most obvious CW complex and cellular map, one   finds the same multivariable zeta function, but now with an $8\times 8$  matrix for $\mathcal F_1$ and a $2\times 2$ matrix for $\mathcal F_0$. 
\end{Rmk} 

\subsection{The function}\label{theFun} 
  Elementary calculation gives the following.
%-----------------------
\begin{Lem}\label{l:zetaFun}     With notation as above, applying Fried's Theorem,  Theorem~\ref{t:frieds}, we have 
\begin{equation}\label{e:multiVarLefschetzAY}
\begin{aligned}  
\zeta_{Q, \hat{h}}(\tau)&= \det(\tau \mathcal F_1-\emph{Id})/(\tau- 1)\\ \\
&=   \dfrac{1 - \tau - \nu \tau^2 - 3 \nu \tau^3 + 
 3 \nu \tau^4 + \nu \tau^5 + \nu^2 \tau^6 - \nu^2 \
\tau^7}{1 -  \tau}\\ \\
 &= 1 - \nu \tau^2 - 
 4 \nu \tau^3 - \nu \tau^4 + \nu^2 \tau^6.  
\end{aligned}
\end{equation}
\end{Lem}
%---------------

\subsection{No section $K$ with $\chi(K) = -2$}\label{ss:notHere}   The following completes the proof of our main result. 

%-----------------------
\begin{Lem}\label{l:noMinTwo}  The suspension flow on $M_{\hat{h}}$ admits no connected section of Euler characteristic equal to $-2$. 
\end{Lem}
%-------------------

\begin{proof}
 From Fried's Theorem, Theorem \ref{t:frieds}, the sections of the suspension flow on $M_{\hat{h}}$ correspond to integral cohomology classes $u_{a,b}$ evaluating $\nu$ and $\tau$ to integers to $a, b$ respectively such that  the induced specialization of $\zeta_{Q, \hat{h}}(\tau)$, 
 \[p_{a,b}(t) :=  t^{2 a + 6 b} - t^{a + 4 b} - 4 t^{a + 3 b}- t^{a + 2 b} + 1\]
 is a polynomial in $t$,  with none of the non-constant terms in the above form becoming constant.  
 
  From the form of $\det(\tau \mathcal F_0 -\text{Id})$,  we must also have  $b>0$; thus,  for $a\ge 0$ one has  $\deg p_{a,b} = 6b+ 2a\ge 6$.    When $a<0$, since $2b>-a$ must hold, so does $2a + 6 b > a+ 4b$.  It now it easily follows that $ \deg p_{a,b} = 2a+ 6b$ here as well. Continuing in this case,  $a+2b\ge 1$ implies that $a+3b \ge 2$ and hence $2a + 6b\ge 4$. 
  
 When $u_{a,b}$ corresponds to a connected section $K_{a,b}$, the Euler characterstic of $K_{a,b}$ equals $\deg p_{a,b}$.   This latter is at least four, and thus the Euler characteristic of any connected section to the  suspension flow on $M_{\hat{h}}$ is at most -4.  The result holds. 
\end{proof} 

\subsection{Excluding other cross sections}   With the multivariable Lefschetz function of $\hat{X}$, and in particular the specializations $p_{a,b}(t)$ in hand, there are various questions that one can begin to address.   As an example, we ask if higher genus pseudo-Anosov examples of \cite{AY} might be found by blowing down cross sections of $M_{\hat{h}}$.   Recall that these occur in each genus $g>3$ each of two singularities, and with stretch factors being the largest root of $x^g - \sum_{i=0}^{g-1} x^i$.     Hence the Euler characteristic is then $-2g$.   To obtain $\deg p_{a,b} =2g$ and respect the restrictions to have a cross section, we must have $(a,b) \in \{(g-3 b, b)\mid 1\le b \le g-1\}$.  A quick search shows that the minimal polynomial of the stretch factor does not divide $p_{a,b}$ in any of these cases when $4 \le g \le 10$.  

\subsection{Shared Alexander polynomial}   We mention an intriguing fact.   The mapping torus of Fried's blown-up toral automorphism and $M_{\hat{h}}$ have the same Alexander polynomial.   Indeed, Fried \cite{Fr85} chooses free generators $\eta, \xi$ for his analog of $H_1(M_{\hat{h}}; \mathbb Z)/\text{Torsion}$, with $\xi$ corresponding to the flow direction, and determines his multivariable Lefschetz zeta function  to be 
\begin{equation}\label{e:friedZetaFun}
1 - (4 + \eta + \eta^{-1}) \xi + \xi^2.
\end{equation}     Clearly, the homomorphism induced from sending the generators $(\eta, \xi)$ to $(\tau, \nu \tau^3)$ is a group isomorphism.  This directly induces an isomorphism of group rings which sends Fried's expression for his multivariable Lefschetz zeta function to \eqref{e:multiVarLefschetzAY}.    As we recalled in Subsection~\ref{s:faceThePositive}, Liu \cite{Liu} shows that Fried's  multivariable Lefschetz zeta function of an  orientable pseudo-Anosov map agrees with the Alexander polynomial of the corresponding mapping torus.      Therefore, after the above innocuous change of variables, the Alexander polynomials of the two mapping tori agree.          In light of this, Theorem~\ref{t:main} can be interpreted as showing that the two blown-up pseudo-Anosov maps cannot belong to the same fibered cone, but not ruling out the possibility that their mapping tori are the same.


\begin{thebibliography}{HKLM}  

\bibitem[A]{A} P. Arnoux,
{\em Un exemple de semi-conjugaison entre un échange d'intervalles et une translation sur le tore},  Bull. Soc. Math. France 116 (1988), no. 4, 489--500 (1989). 
%   {\em \'Echanges d'intervalles et flots sur les surfaces}, pp. 5--38 in
%     \emph{Ergodic theory (Sem., Les Plans-sur-Bex, 1980)},
 %    Monograph. Enseign. Math., 29, Univ. Gen\`eve, Geneva, 1981. 

%\bibitem[A2]{A2} \bysame,
 % Th\`ese de 3$^{e}$ cycle, Universit\'e de Reims, 1981. 
  
  
  
\bibitem[ABB]{ABB} P.~Arnoux, J.~Bernat, and X.~Bressaud,
 {\em Geometrical models for substitutions}, Exp. Math. 20 (2011), 97--127.  

%\bibitem[AR]{AR} P. Arnoux and G.  Rauzy,  
%{\em Repr\'esentation g\'eom\'etrique de suites de complexit\'e $2n+1$},  
%Bull. Soc. Math. France 119 (1991), no. 2, 199--215. 

\bibitem[AY]{AY} P.~Arnoux and J.-C.~Yoccoz, {\em Construction de diff\'eomorphismes pseudo-Anosov}, 
C.R. Acad. Sci. Paris Sr. I Math. 292 (1981), no. 1, 75--78.%

\bibitem[BSW]{BSW}
M.~Bainbridge,  J.~Smillie and B.~Weiss, {\em Horocycle dynamics: new invariants and eigenform loci in the stratum $\mathcal H(1,1)$},  Mem. Amer. Math. Soc. 280 (2022), no. 1384.

\bibitem[B]{Bowman}
%:
J.~Bowman, {\em The complete family of Arnoux-Yoccoz surfaces}, Geom. Dedicata 164 (2013), 113--130.

  
\bibitem[Boy]{Boyland}
P.~Boyland,  {\em Topological methods in surface dynamics},  Topology Appl. 58 (1994), no. 3, 223--298. 
 

\bibitem[Boy2]{Boyland2}
\bysame,  {\em Isotopy stability of dynamics on surfaces}, In:   Geometry and topology in dynamics (Winston-Salem, NC, 1998/San Antonio, TX, 1999), 17--45,
Contemp. Math., 246, Amer. Math. Soc., Providence, RI, 1999. 

\bibitem[FLP]{FLP}
 A.~Fathi, F.~Laudenbach, V.~Po\'enaru,   Travaux de
Thurston sur les surfaces.  Ast\'erisque, 66--67,  Soci\'et\'e math\'ematique de France,
1979.      English translation:  Thurston’s Work on Surfaces, translated by
D.~M.~Kim and D.~Margalit. Princeton University Press, Princeton, NJ, 2012.
 
\bibitem[Fr79]{Fr79} D.~Fried,  {\em Fibrations over $S^1$ with pseudo-Anosov monodromy},  in \cite{FLP}, op. cit.,  pp. 251--266.

\bibitem[Fr82a]{Fr82GeomCrossTop} \bysame, {\em The geometry of cross sections to flows}, Topology, 21(4), 353--371, 1982.

\bibitem[Fr82b]{Fr82FlowEqui} \bysame, {\em Flow equivalence, hyperbolic systems and a new zeta function for flows},  Comment. Math. Helv. 57 (1982), no. 2, 237--259.

\bibitem[Fr85]{Fr85} \bysame, 
{\em Growth rate of surface homeomorphisms and flow equivalence}, 
Ergodic Theory and Dyn. Sys.,   (1985) 5, 539--563.

\bibitem[H]{Hatcher} 
 A.~Hatcher,  Algebraic topology. Cambridge University Press, Cambridge, 2002.
 
\bibitem[HW]{HooperWeiss}  
W.~P.~Hooper, and B.~Weiss, 
{\em Rel leaves of the Arnoux-Yoccoz surfaces},
With an appendix by Lior Bary-Soroker, Mark Shusterman, and Umberto Zannier.
Selecta Math. (N.S.) 24 (2018), no. 2, 875--934.
 
\bibitem[Hu]{Hteich2} 
J.~Hubbard, 
Teichm\"uller theory and applications to geometry, topology, and dynamics, Vol. 2:
Surface homeomorphisms and rational functions. Matrix Editions, Ithaca, NY, 2016.

\bibitem[HM]{HM} J. Hubbard and H. Masur, 
{\em Quadratic differentials and foliations},  Acta Math. 142 (1979), no. 3--4, 221--274.
 
\bibitem[HLM]{HLM} P.~Hubert, E.~Lanneau and M.~M\"oller, {\em The Arnoux-Yoccoz Teichm\"uller disc},  Geom. Funct. Anal. 18 (2009), no. 6, 1988--2016.

\bibitem[J]{Jiang} B.~Jiang, Lectures on Nielsen Fixed Point Theory, Contemp. Math. 14, Amer. Math. Soc., Province,
RI, 1983.
 
\bibitem[L]{Landry} M.~ Landry, {\em Stable loops and almost transverse surfaces},  Groups Geom. Dyn. 17 (2023), no. 1, 35--75.
   
\bibitem[Lann]{Lanneau} E.~Lanneau,  
{\em Dis-moi un pseudo-Anosov. (French) [Tell me a pseudo-Anosov]},
Gaz. Math. No. 151 (2017), 52--57. 

\bibitem[LT]{LT}  E. Lanneau and J.-C. Thiffeault, 
{\em On the minimum dilatation of pseudo-Anosov homeomorphisms on surfaces of small genus},  Ann. Inst. Fourier (Grenoble) 61 (2011), no. 1, 105--144.

\bibitem[Li]{Liu}
Y.~Liu,
{\em Virtual homological spectral radii for automorphisms of surfaces},
J. Amer. Math. Soc. 33 (2020), no. 4, 1167--1227. 

\bibitem[Li2]{Liu2} 
\bysame,
{\em Mapping classes are almost determined by their finite quotient actions},
Duke Math. J. 172 (2023), no. 3, 569--631.

\bibitem[LS]{LiechtiStrenner} L.~Liechti and B.~Strenner, 
{\em The Arnoux-Yoccoz mapping classes via Penner's construction},
Bull. Soc. Math. France 148 (2020), no. 3, 383--397. 
 
\bibitem[LPV]{LPV} J.~H.~Lowenstein, G.~Poggiaspalla, and F.~Vivaldi,
 {\em Interval exchange transformations over algebraic number fields: the cubic Arnoux-Yoccoz model}, Dyn. Syst. 22 (2007), 73--106. 
 
\bibitem[Mc]{McPoly} C.~T.~McMullen, {\em Polynomial invariants for fibered 3-manifolds and Teichmüller geodesics for foliations},  Ann. Sci. \'Ecole Norm. Sup. (4) 33 (2000), no. 4, 519--560. 

\bibitem[Mc2]{McAlexander} \bysame, {\em The Alexander polynomial of a 3-manifold and the Thurston norm on cohomology}, Ann. Sci. \'Ecole Norm. Sup. (4) 35 (2002), no. 2, 153--171.

\bibitem[P]{Parlak} A.~Parlak, {\em The taut polynomial and the Alexander polynomial},  J. Topology, Volume 16, Issue 2 (2023),  720--756.
 
 
\bibitem[S]{S} J.-P. Serre, {\em 	Homologie singuli\`ere des espaces fibr\'es},
 Annals of Math., 2nd Ser., Vol. 54, No. 3. (Nov., 1951), pp. 425--505.

\bibitem[Sh]{Shub} M.~Shub,  Stabilit\'e globale des syst\`emes dynamiques. Ast\'erisque, 56. Société Mathématique de France, Paris, 1978.
                             English translation:  Global stability of dynamical systems.  Translated  by Joseph Christy. Springer-Verlag, New York, 1987. 

\bibitem[Str]{Str}  B.~Strenner, {\em Lifts of pseudo-Anosov homeomorphisms of nonorientable surfaces have vanishing SAF invariant},
Math. Res. Lett. 25 (2018), no. 2, 677--685.

\bibitem[Th1]{ThurstonNorm} W.~Thurston, {\em A norm for the homology of 3-manifolds}, Mem. Amer. Math. Soc. 339 (1986),
99--130.

\bibitem[Th2]{ThurstonPA} \bysame {\em On the geometry and dynamics of diffeomorphisms of surfaces},  Bull. Amer. Math.
Soc. (N.S.) 19 (1988), 417--431.

\bibitem[T]{T} V. Turaev, 
Introduction to combinatorial torsions.
Lectures Math. ETH Zürich, 
Birkhäuser Verlag, Basel, 2001
 
 
\end{thebibliography}
\end{document}